\documentclass[11pt,a4paper]{article}
\usepackage[left=2cm, top=2.45cm,bottom=2.45cm,right=2cm]{geometry}
\usepackage{mathtools,amssymb,amsthm,mathrsfs,calc,graphicx,xcolor,cleveref,dsfont,tikz,pgfplots,bm,url}
\usepackage[british]{babel}
\usepackage{amsfonts}              
\usepackage[T1]{fontenc}
\usepackage{enumitem}
\numberwithin{equation}{section}
\numberwithin{figure}{section}

\usepackage{float}
\restylefloat{table}

\theoremstyle{plain}
\newtheorem{theorem}{Theorem}[section]
\newtheorem{lemma}[theorem]{Lemma}
\newtheorem{corollary}[theorem]{Corollary}
\newtheorem{proposition}[theorem]{Proposition}

\theoremstyle{remark}
\newtheorem{remark}[theorem]{Remark}

\newcommand{\dint}{\textup{d}}

\newcommand{\cl}{\mathop{\mathrm{cl}}\nolimits}

\def\BB{\mathbb{B}}
\def\CC{\mathbb{C}}
\def\EE{\mathbb{E}}
\def\NN{\mathbb{N}}
\def\PP{\mathbb{P}}
\def\RR{\mathbb{R}}
\def\SS{\mathbb{S}}

\def\cC{\mathcal{C}}
\def\cD{\mathcal{D}}
\def\cF{\mathcal{F}}
\def\cL{\mathcal{L}}

\newcommand{\dd}{{\rm d}}
\newcommand{\conv}{\mathop{\mathrm{conv}}\nolimits}
\newcommand{\pow}{\mathop{\mathrm{pow}}\nolimits}
\newcommand{\dist}{\mathop{\mathrm{dist}}\nolimits}
\newcommand{\aff}{\mathop{\mathrm{aff}}\nolimits}
\newcommand{\proj}{\mathop{\mathrm{proj}}\nolimits}
\newcommand{\inter}{\operatorname{int}}
\newcommand{\Vol}{\operatorname{Vol}}
\newcommand{\VV}{\operatorname{\mathbb{V}ar}}

\renewcommand{\Re}{\operatorname{Re}}

\begin{document}

\title{\bfseries The $\beta$-Delaunay tessellation III:\\ Kendall's problem and limit theorems in high dimensions}

\author{Anna Gusakova\footnotemark[1],\; Zakhar Kabluchko\footnotemark[2],\; and Christoph Th\"ale\footnotemark[3]}

\date{}
\renewcommand{\thefootnote}{\fnsymbol{footnote}}
\footnotetext[1]{Ruhr University Bochum, Germany. Email: anna.gusakova@rub.de}

\footnotetext[2]{M\"unster University, Germany. Email: zakhar.kabluchko@uni-muenster.de}

\footnotetext[3]{Ruhr University Bochum, Germany. Email: christoph.thaele@rub.de}

\maketitle

\begin{abstract}
The $\beta$-Delaunay tessellation in $\RR^{d-1}$ is a generalization of the classical Poisson-Delaunay tessellation. As a first result of this paper we show that the shape of a weighted typical cell of a $\beta$-Delaunay tessellation, conditioned on having large volume, is close to the shape of a regular simplex in $\RR^{d-1}$. This generalizes earlier results of Hug and Schneider about the typical (non-weighted) Poisson-Delaunay simplex. Second, the asymptotic behaviour of the volume of weighted typical cells in high-dimensional $\beta$-Delaunay tessellation is analysed, as $d\to\infty$. In particular, various high dimensional limit theorems, such as quantitative central limit theorems as well as moderate and large deviation principles, are derived.  

\bigskip

\end{abstract}

\section{Introduction}

The classical Poisson-Delaunay tessellation in $\RR^{d-1}$, which is widely used in stochastic modelling of spatial random structures \cite{AurenhammerKlein,BlaszEtAl,Haenggi,OkabeEtAl,PreparataShamos,SKM}, is constructed as follows. The starting point is a stationary Poisson point process $\eta$ in $\RR^{d-1}$ of intensity $\gamma>0$. For $d$ distinct points $x_1,\ldots,x_d$ of $\eta$ we consider the almost surely uniquely determined ball having $x_1,\ldots,x_d$ on its boundary. If this ball contains no points of $\eta$ in its interior, the convex hull $\conv(x_1,\ldots,x_d)$ of $x_1,\ldots,x_d$ becomes a so-called Delaunay simplex. The collection of all Delaunay simplices is a stationary random simplicial tessellation of $\RR^{d-1}$,  which is called Poisson-Delaunay tessellation. To explain the construction of a $\beta$-Delaunay tessellation, we fix a parameter $\beta>-1$ and consider a Poisson point process $\eta_\beta$ in the height- or time-augmented space $\RR^{d-1}\times\RR_{+}$ whose intensity measure $\mu_\beta$ is the Lebesgue measure in the spatial coordinate $\RR^{d-1}$ and has density $h^\beta$ in the height or time coordinate $\RR_{+}$, up to an intensity parameter $\gamma c_{d,\beta}$. Formally, $\mu_\beta$ satisfies
$$
\int_{\RR^{d-1}\times\RR_{+}}f(v,h)\,\mu_\beta(\dint(v,h)) = \gamma c_{d,\beta}\int_{\RR^{d-1}}\int_{\RR_{+}}f(v,h)\,h^\beta\,\dint h\dint v
$$
for every non-negative measurable function $f:\RR^{d-1}\times\RR_{+}\to\RR$. Now, given $d$ distinct points $x_1=(v_1,h_1),\ldots,x_d=(v_d,h_d)$ of $\eta_\beta$, there is an almost surely unique translate of the standard downward paraboloid $\left\{(v,h)\in\RR^{d-1}\times\RR\colon h\leq -\|v\|^2\right\}$ containing the points $x_1,\ldots,x_d$ on its boundary. The random simplex $\conv(v_1,\ldots,v_d)$ in $\RR^{d-1}$, which is formed by the spatial projection of the points $x_1,\ldots,x_d$, is a $\beta$-Delaunay simplex if and only if the interior of the downward paraboloid determined by $x_1,\ldots,x_d$ does not contain any point of $\eta_{\beta}$. The collection of all $\beta$-Delaunay simplices is again a stationary random simplicial tessellation, which is called the $\beta$-Delaunay tessellation in $\RR^{d-1}$ and was introduced and studied in part I of this series of articles \cite{Part1}. {As alreday mentioned in \cite{Part1}, the $\beta$-Delaunay tessellation describes the local asymptotic structure of so-called beta random polytopes in the $d$-dimensional unit ball close to its boundary. Namely, after a suitable rescaling the boundary of the unit ball locally `looks' like $\RR^{d-1}$ and the boundary of the beta random polytope, projected to the unit sphere, locally `looks' like the $\beta$-Delaunay tessellation. We also mention that the class of beta random polytopes has recently been given special attention in the stochastic geometry literature as they provide a common link between a number of stochastic geometry models, see e.g.\ \cite{KTT,KTZ20} and the references given therein.}

In the focus of this paper is the $\nu$-weighted typical cell $Z_{\beta,\nu}$ of the $\beta$-Delaunay tessellation, which is almost surely a $(d-1)$-dimensional simplex, the so-called $\nu$-weighted typical $\beta$-Delaunay simplex, where $\nu>-1$ is a weight parameter. We recall that, on an intuitive level, the typical cell is a random simplex picked uniformly at random from the infinite collection of all $\beta$-Delaunay simplices, regardless of size and shape. If the cells are weighted according to the $\nu$th power of their volume and a simplex is now chosen according to these weights, one arrives at the $\nu$-weighted typical $\beta$-Delaunay simplex. We will make this mathematically rigorous using the concept of Palm distribution in Section \ref{sec:TypCellDef} below. Let us remark that for $\nu=0$ we get back the typical $\beta$-Delaunay simplex, while for $\nu=1$ the $\nu$-weighted typical cell has, up to translation, the same distribution as the almost surely uniquely determined cell containing the origin of $\RR^{d-1}$ (i.e., the zero cell).

In the first part of this paper we study the following question: What is the 'shape' of the $\nu$-weighted typical cell $Z_{\beta,\nu}$, conditionally on the event that its volume is large? For the stationary and isotropic Poisson line tessellation in the plane an analogous question goes back to D.G.\ Kendall (see, for example, the preface of \cite{SKM}) and has first been studied by Kovalenko \cite{Kovalenko}. Subsequently, this has triggered substantial interest in stochastic geometry and a number of variations of Kendall's problem have been investigated for various random tessellation models, we refer to \cite{HugSchneiderGAFA} for a rather general result, and the survey articles \cite{HugReitznerIntroSG,SchneiderSurvey} for an overview. {Most relevant in our context are the works of Hug and Schneider \cite{HugSchneiderDelaunay,HSDelaunay2}, where (general versions of) Kendall's problem has been studied for the typical cell of a classical Poisson-Delaunay tessellation in $\RR^{d-1}$.} It has been proven there that the shape of the typical Poisson-Delaunay simplex, conditionally on having a large volume, is close to that of a regular simplex in $\RR^{d-1}$. Our first result shows, in a quantitative way, that the same phenomenon can be observed for the broad class of $\nu$-weighted typical $\beta$-Delaunay simplices $Z_{\beta,\nu}$. For example, if $\rho(Z_{\beta,\nu})$ stands for a suitable measurement for the distance of $Z_{\beta,\nu}$ to a $(d-1)$-dimensional simplex, which we formally introduce in Section \ref{sec:KendallBeta} below, we show in Theorem \ref{thm:Kendall} that
$$
\PP(\rho(Z_{\beta,\nu})\geq\varepsilon\,|\,\Vol(Z_{\beta,\nu})\geq a) \leq C\,\exp\Big\{{-c\,\varepsilon^2\,a^{d+1+2\beta\over d-1}}\Big\}
$$
for any $\varepsilon\in(0,1)$ and suitable constants $C,c\in(0,\infty)$. In particular, formally taking $\beta=-1$ and $\nu=0$, we recover the main result of \cite{HugSchneiderDelaunay} for the typical Poisson-Delaunay simplex. {We also study in Theorem \ref{thm:VolumeTails} the tail behaviour of the random variable $\Vol(Z_{\beta,\nu})$ at zero and infinity, which turns out to be describable by a suitable polynomial and an exponential function, respectively.}

In the second part of this paper we study the logarithmic volume of the $\nu$-weighted typical $\beta$-Delaunay simplex $Z_{\beta,\nu}$ in high dimensions, that is, as the dimension $d$ of the ambient space increases, namely $d\to\infty$. The motivation for such a study is driven by the fact that probabilistic limit theorems for convex bodies in high dimensions is a central theme in the branch of mathematics called Asymptotic Geometric Analysis. The behaviour of the logarithmic volume of \textit{random} convex bodies and especially of random simplices in high dimensions has recently been studied in stochastic geometry in \cite{OberwolfachSimplices,GKT17,GusakovaThaeleDelaunay}. We continue this line of research by providing central limit theorems as well as moderate and large deviation principles for $Y_{\beta, \nu, d}:=\log\Vol(Z_{\beta,\nu})$, as $d\to\infty$. For example, for fixed $\nu\geq -1$, $\beta>-1$ and as $d\to\infty$, we show that
\begin{align*}
	\EE Y_{\beta,\nu,d}&=-d\log d+{3+2\nu\over 4}\log d+{d\over 2}+O(1),\\
	\VV Y_{\beta,\nu,d}&={1\over 2}\log d+C_{\nu}+O(1/d),
\end{align*}
for some explicit constant $C_\nu\in(0,\infty)$ only depending on $\nu$, see Corollary \ref{cor:CumulantsTypicalBeta}. Moreover, putting
$$
\widetilde Y_{\beta, \nu, d}:=\frac{Y_{\beta, \nu, d}-\EE Y_{\beta, \nu, d}}{\sqrt{\VV Y_{\beta, \nu, d}}},
$$
the following quantitative central limit theorem holds, where $c\in(0,\infty)$ is some constant depending on the model parameters $\beta$ and $\nu$, see Theorem \ref{thm:Main}:
$$
\sup_{y\in\RR}|\PP(\widetilde Y_{\beta, \nu, d}\leq y)-\Phi(y)| \leq {c\over\sqrt{\log d}}.
$$
We emphasize that in the special case $\beta=-1$, which corresponds to the classical Poisson-Delaunay tessellation, this covers previous results from \cite{GusakovaThaeleDelaunay}. We also remark in this context that central limit theorems for other functionals of $\beta$-Delaunay tessellations for fixed space dimensions but in increasing observation windows will be derived in part IV of this paper.

\bigskip

The remaining parts of the text are structured as follows. In Section \ref{sec:Preliminaries} we recall some necessary background material in order to make this paper self-contained. In particular, we rephrase there the construction of the $\beta$-Delaunay tessellation as well as definition of typical weighted $\beta$-Delaunay simplices. Kendall's problem for such random simplices is studied in Section \ref{sec:KendallBeta}, whereas in Section \ref{sec:LimitTheorems} we concentrate on central limit theorems and moderate and large deviations for the logarithmic volume of typical weighted $\beta$-Delaunay simplices in high dimensions.

\section{Preliminaries about the $\beta$-Delaunay tessellation}\label{sec:Preliminaries}

\subsection{Notation and set-up}

In this paper we use the following notation. Given a set $A\subset\RR^{d-1}$, $d\ge 2$ we denote by $\conv(A)$ its convex hull and by ${\rm int}(A)$ its topological interior. A centred closed Euclidean unit ball in $\RR^{d-1}$ is denoted by $\BB^{d-1}$ and its volume is given by
$
\kappa_d:=\frac{\pi^{d/2}}{\Gamma(1+{d\over 2})}.
$
By $\RR_{+}:=[0,\infty)$ we denote the set of all non-negative real numbers. We shall represent points $x\in\RR^d$ in the form $x=(v,h)$ with $v\in\RR^{d-1}$ (called \textit{spatial coordinate}) and $h\in\RR$ (called \textit{height}, \textit{weight} or \textit{time coordinate}). 

We start by recalling the definition and the main properties of $\beta$-Delaunay tessellations. The first description of $\beta$-Delaunay tessellation was given in \cite{Part1}. Fundamental facts about Poisson point processes and tessellations, which we will omit here, can be found in \cite{SW, LP}. Consider a Poisson point process $\eta_{\beta}$ in $\RR^{d-1}\times \RR_{+}$ with intensity measures having density
$$
\gamma\,c_{d,\beta}\,h^{\beta},\qquad c_{d,\beta}={\Gamma({d\over 2}+\beta+1)\over \pi^{d\over 2}\Gamma(\beta+1)},\, \, \gamma > 0,
$$
with respect to the Lebesgue measure on $\RR^{d-1}\times\RR_{+}$.

According to \cite{Part1} there exist two alternative ways to construct $\beta$-Delaunay tessellations based on the Poisson processes $\eta_\beta$. The first construction uses the notion of Laguerre tessellations and it is more convenient for defining the tessellation and investigating the properties of the typical cell. The second construction is defined via the paraboloid hull process, introduced in \cite{CSY13, SY08}, and it is used for studying the convergence of $\beta$-tessellations as $\beta\to\infty$ (see \cite{Part2}) and its mixing properties in part IV of this paper. We will consider only the approach based on Laguerre tessellations here, for more details regarding the second approach we refer reader to \cite[Section 3.5]{Part1} and \cite[Section 3.1]{Part2}.  

\subsection{Construction}\label{sec:Laguerre_tess}

One of the most well studied type of tessellations is the classical Voronoi tessellation, see, for example, \cite{OkabeEtAl,SKM,SW}. A Laguerre tessellation can be considered as a generalized (or weighted) version of a Voronoi tessellation and was intensively studied in \cite{LZ08, Ldoc, Sch93}. In this subsection we only briefly recall some facts about Laguerre tessellations. For more details we refer the reader to part I of this paper, especially to \cite[Section 3.2 - 3.4]{Part1}.

The construction of a Laguerre diagram is based on the notion of a power function. For $v,w \in \RR^{d-1}$ and $W\in\RR$ we define the power of $w$ with respect to the pair $(v,W)$ as
\[
\pow (w,(v,W)):=\|w-v\|^2+W.
\]
In this situation $W$ is refereed as a weight of the point $v$. Let $X\subset\RR^{d-1}\times\RR$ be a countable set of marked points in $\RR^{d-1}$ such that $\min_{(v,W)\in X}\pow(w,(v,W))$ exists for each $w\in\RR^{d-1}$. Then the Laguerre cell of $(v,W)\in X$ is defined as
\[
C((v,W),X):=\{w\in\RR^{d-1}\colon \pow(w,(v,W))\leq \pow(w,(v',W'))\text{ for all }(v',W')\in X\}.
\]
We emphasize that it is not necessarily the case that a Laguerre cell is non-empty or that it contains interior points. The collection of all Laguerre cells of $X$ having non-vanishing topological interior is called the Laguerre diagram
\[
\cL(X):=\{C((v,W),X)\colon (v,W)\in X, {\rm int}(C((v,W),X))\neq\varnothing\}.
\]
It should be mentioned that Laguerre diagram is not necessarily a tessellation, since the latter strongly depends on the geometric properties of the set $X$. However, it was shown in \cite{Part1} that $\cL(\eta_{\beta})$ for $\beta>-1$ is indeed random normal tessellation. 

Let $\cL^*(\eta_{\beta})$ be the dual tessellation of $\cL(\eta_{\beta})$. This tessellation arises from $\cL(\eta_{\beta})$ by including for distinct points $x_1=(v_1,h_1),\ldots,x_d=(v_d,h_d)$ of $\eta_{\beta}$ the simplex $\conv(v_1,\ldots,v_d)$ in $\cL^*(\eta_{\beta})$ if and only if the Laguerre cells corresponding to $(v_1, h_1),\ldots,(v_d, h_d)$ all have non-empty interior and share a common point. In our case $\cL^*(\eta_{\beta})$ is almost surely a stationary random simplicial tessellation, and moreover it can be regarded as a Laguerre tessellation of the random set 
\begin{equation}\label{eq:ApexProcess}
\eta_{\beta}^{*}:=\left\{(z,K_{z})\in\RR^{d-1}\times\RR\colon z\in \cF_0(\cL(\eta_{\beta}))\right\},
\end{equation}
where $\cF_0(\cL(\eta_{\beta}))$ denote a set of vertices of the tessellation $\cL(\eta_{\beta})$ and $K_{z}$ is a constant, such that $z\in \cF_0(C((v,h),\eta_{\beta}))$ if and only if $\pow(z, (v,h))=K_{z}$ and there is no $(v,h)\in\eta_{\beta}$ with $\pow(z, (v,h))<K_{z}$, see \cite{Part1} for details. The random tessellation $\cD_\beta:=\cL^*(\eta_{\beta})$, $\beta>-1$ is called the \textbf{$\beta$-Delaunay tessellation} in $\RR^{d-1}$, see Figure \ref{fig:betatess} for two simulations.

\begin{figure}
\centering
\includegraphics[width=0.45\columnwidth]{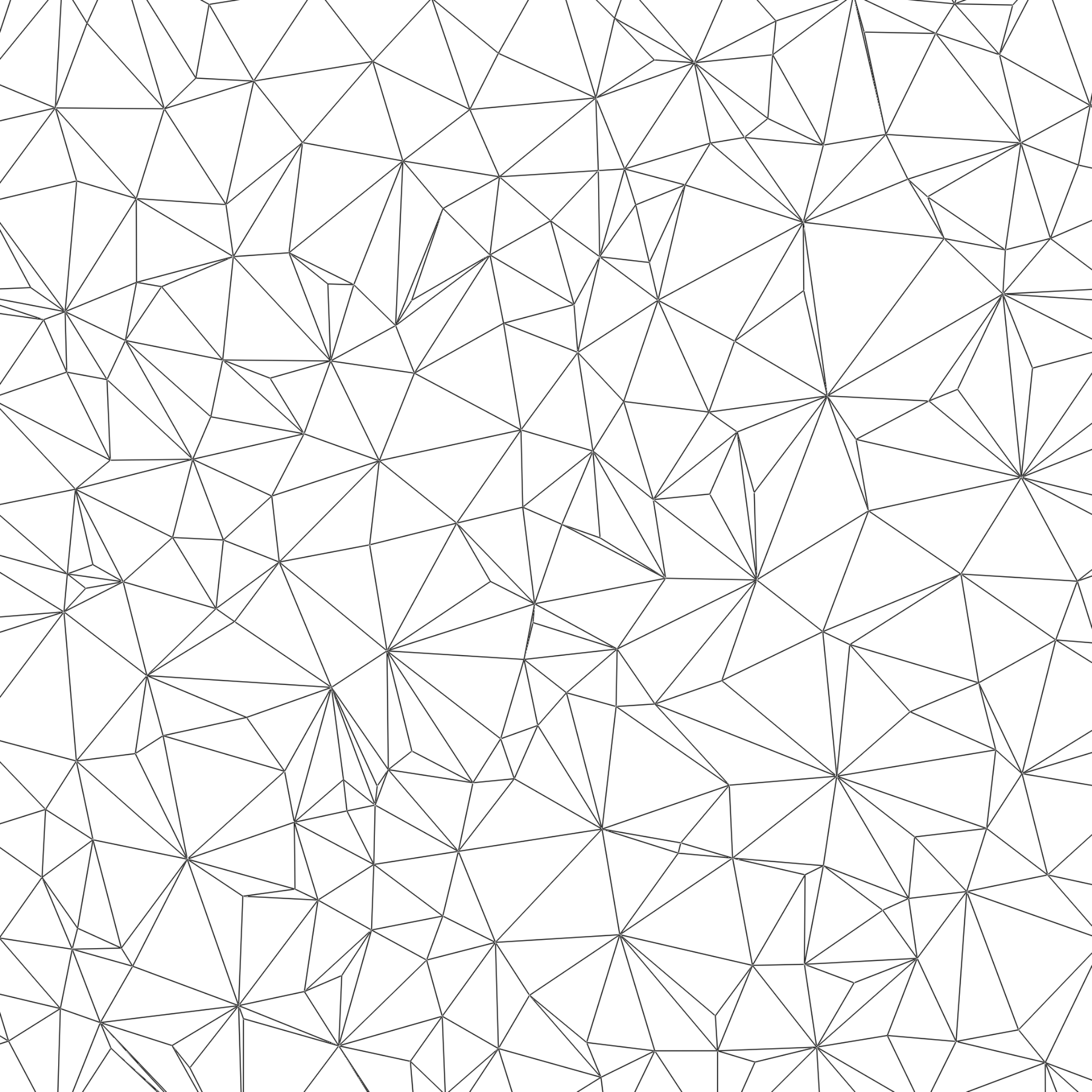}\qquad
\includegraphics[width=0.45\columnwidth]{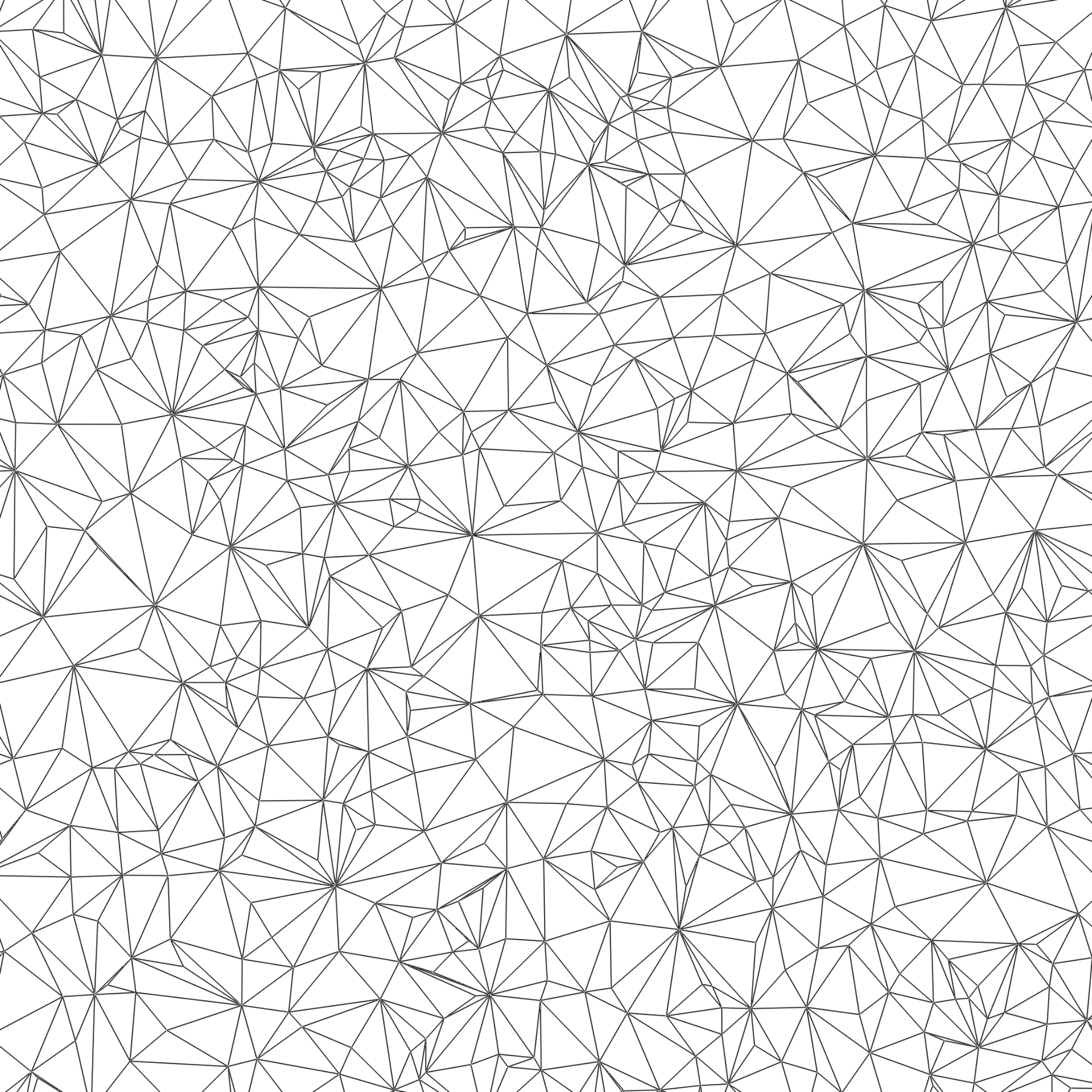}
\caption{Simulation of $\beta$-Delaunay tessellations in the plane with $\beta=5$ (left) and $\beta=15$ (right).}
\label{fig:betatess}
\end{figure}

\subsection{The $\nu$-weighted typical cell}\label{sec:TypCellDef}

The formal definition of the $\nu$-weighted typical cell for the family of random tessellation $\cL^{*}(\xi)$, where $\xi$ is a Poisson point process satisfying some natural assumptions was given in \cite{Part1}. These requirements are met, for example, by the Poisson point processes $\eta_\beta$ for $\beta>-1$. The definition relies on the concept of generalized centre functions and Palm calculus for marked point processes (see \cite[p.\ 116]{SW} and \cite[Section 4.3]{SWGerman}), which is a standard approach in this case. In this subsection we will briefly recall this definition as well as some results regarding the distribution of the $\nu$-weighted typical cells of the $\beta$-Delaunay tessellation. More details can be found in \cite[Section 4]{Part1}.

As explained in the previous section, the random tessellation $\cD_{\beta}$ coincides with the Laguerre tessellation of the random set $\eta_\beta^{*}$, described by \eqref{eq:ApexProcess}. 
We consider the random marked point process
\[
\chi_{\beta}:=\sum\limits_{(v,h)\in\eta_\beta^{*}}\delta_{(v,M)},\qquad M:=C((v,h), \eta_\beta^{*})-v,
\]
in $\RR^{d-1}$, whose marks are the associated and suitably centred Laguerre cells. For a given parameter $\nu\in\RR$ we now define a probability measure $\PP_{\beta,\nu}$ on the space $\cC'$ of non-empty compact subsets of $\RR^{d-1}$ as follows:
\[
\PP_{\beta,\nu}(\,\cdot\,) := {1\over \lambda_{\beta,\nu}}\EE\sum_{(v,M)\in\chi_{\beta}}{\bf 1}(M\in\,\cdot\,){\bf 1}_{[0,1]^{d-1}}(v)\Vol(M)^\nu,
\]
where  $\lambda_{\beta,\nu}\in[0,\infty]$ is the normalizing constant given by
\[
\lambda_{\beta,\nu}:= \EE\sum_{(v,M)\in\chi_{\beta}}{\bf 1}_{[0,1]^{d-1}}(v)\Vol(M)^\nu.
\]
Note that $\lambda_{\beta,\nu}$ might be infinite for some values of $\nu$. For any $\nu$ with $\lambda_{\beta,\nu}<\infty$, a random simplex $Z_{\beta,\nu}$ with distribution $\PP_{\beta,\nu}$ is called the {\bf  $\Vol^{\nu}$-weighted} (or just {\bf $\nu$-weighted}) {\bf typical cell} of the tessellation $\cD_{\beta}$. The following two special cases are of particular interest:
\begin{itemize}
\item[(i)]  $Z_{\beta,0}$ coincides with the classical typical cell of $\cD_{\beta}$;
\item[(ii)]  $Z_{\beta,1}$ coincides with the volume-weighted typical cell of $\cD_{\beta}$, which has the same distribution as the almost surely uniquely determined cell containing the origin, up to translation.
\end{itemize}

The next result is an explicit description of the distribution of the $\nu$-weighted typical cell of a $\beta$-Delaunay tessellation taken from  \cite[Theorem 4.5]{Part1}.

\begin{lemma}[Distribution of the $\nu$-weighted typical cell]\label{theo:typical_cell_stoch_rep}
Fix $d\geq 2$, $\nu\ge-1$, $\beta>-1$ and $\gamma >0$. Then for any Borel set $A\subset \cC'$ we have that
\begin{align*}
\PP_{\beta,\nu}(A)
&=
\alpha_{d,\beta,\nu}\int_{(\BB^{d-1})^d}\dd y_1\ldots \dd y_d \, \int_{0}^{\infty}\dd r\,{\bf 1}_A(\conv(ry_1,\ldots,ry_d)) r^{2d\beta+d^2+\nu(d-1)}\notag\\
&\qquad\times e^{-m_{d,\beta} r^{d+1+2\beta}}
\Delta_{d-1}(y_1,\ldots,y_d)^{\nu+1}\prod_{i=1}^d(1-\|y_i\|^2)^{\beta},
\end{align*}
where $\Delta_{d-1}(y_1,\ldots,y_d)$ is the volume of $\conv(y_1,\ldots,y_d)$ and $\alpha_{d,\beta,\nu}$ and $m_{d,\beta}$ are constants given by
\begin{align}
\label{eq:DefMdbeta} m_{d,\beta}&={\gamma\Gamma({d\over 2})\over 2\sqrt{\pi}\Gamma({d+1\over 2})},\\
\alpha_{d,\beta,\nu}&=\pi^{d(d-1)\over 2}\,{((d-1)!)^{\nu+1}(d+1+2\beta)\Gamma({d(d+\nu+2\beta)-\nu+1\over 2})\over \Gamma({d(d+\nu+2\beta)\over 2}+1)\Gamma(d+{(\nu-1)(d-1)\over d+2\beta+1})}\Big({\gamma\,\Gamma({d\over 2}+\beta+1)\over \sqrt{\pi}\Gamma({d+1\over 2}+\beta+1)}\Big)^{d+{(\nu-1)(d-1)\over d+2\beta+1}}\notag\\
&\qquad\qquad\times{\Gamma({d+\nu\over 2}+\beta+1)^d\over \Gamma(\beta+1)^d}\prod\limits_{i=1}^{d-1}{\Gamma({i\over 2})\over \Gamma({i+\nu+1\over 2})}.\notag
\end{align}
\end{lemma}

\begin{remark}\label{rem:rep_typical}
In more probabilistic terms, the $\nu$-weighted typical cell of the $\beta$-Delaunay tessellation $\cD_\beta$ has the same distribution as the random simplex $\conv(RY_1,\ldots,RY_d)$, where
\begin{enumerate}
\item[(a)] $R$ is a random variable whose density is proportional to $r^{2d\beta+d^2+\nu(d-1)}e^{-m_{d,\beta} r^{d+1+2\beta}}$ on $(0,\infty)$;
\item[(b)] $(Y_1,\ldots,Y_d)$ are $d$ random points in the unit  ball $\BB^{d-1}$ whose joint density is proportional to
$$
\Delta_{d-1}(y_1,\ldots,y_d)^{\nu+1}   \prod\limits_{i=1}^d(1-\|y_i\|^2)^{\beta},
\qquad y_1\in \BB^{d-1},\ldots, y_d\in \BB^{d-1};
$$
\item[(c)] $R$ is independent of $(Y_1,\ldots,Y_d)$.
\end{enumerate}
\end{remark}

\begin{remark}
It follows from the proof of Theorem 4.5 in \cite{Part1} that the normalizing constant $\lambda_{{\beta},\nu}$ is finite for any $\nu \ge -1$. It was also conjectured that it is possible to enlarge the diapason of possible values for $\nu$ to $\nu>-2$.
\end{remark}

\begin{remark}\label{rm:BetaToPoisson}
Let us point out that in the limiting case $\beta\to -1$ the beta distribution with density $c_{d-1,\beta}(1-\|x\|^2)^{\beta}{\bf 1}_{\BB^{d-1}}(x)$ weakly converges to the uniform distribution on the unit sphere $\SS^{d-2}$. Thus, $\PP_{\beta,\nu}$ for fixed $\nu\ge-1$ and $\gamma>0$ weakly converges to a probability measure $\PP_{-1,\nu}$, which coincides with the the distribution of the $\nu$-weighted typical cell of Poisson-Delaunay tessellation in $\RR^{d-1}$ corresponding to the intensity $\gamma\omega_d^{-1}$ of underlying Poisson point process, where $\omega_d$ is the surface area of unit $(d-1)$-dimensional sphere, see \cite[Theorem 2.3]{GusakovaThaeleDelaunay} for general $\nu$ and \cite[Theorem 10.4.4]{SW} for the case $\nu=-1$.
\end{remark}

The following lemma contains an explicit formula for the moments of the random variables $\Vol(Z_{\beta,\nu})$ and is taken from \cite[Theorem 5.1]{Part1}. 
We will use it in Section \ref{sec:LimitTheorems} when we study the asymptotic behaviour of the log-volume of $Z_{\beta,\nu}$ in high dimensions.

\begin{lemma}[Volume moments of the $\nu$-weighted typical cell]\label{lm:VolumeMomBeta}
Let $Z_{\beta,\nu}$ be the $\nu$-weighted typical cell of a $\beta$-Delaunay tessellation with $\beta\geq -1$ and $\nu\ge-1$. Then, for any $s\ge-\nu-1$, we have
\begin{align*}
\EE \Vol(Z_{\beta,\nu})^s &= {1\over ((d-1)!)^s}\Big({ \sqrt{\pi}\Gamma({d+1\over 2}+\beta+1)\over \gamma\Gamma({d\over 2}+\beta+1)}\Big)^{{s(d-1)\over d+2\beta+1}}{\Gamma({d(d+2\beta)+\nu(d-1)+1\over 2})\over\Gamma({d(d+2\beta)+(\nu+s)(d-1)+1\over 2})}{\Gamma({d(d+\nu+s +2\beta)\over 2}+1)\over\Gamma({d(d+\nu +2\beta)\over 2}+1)}\\
&\qquad\times{\Gamma(d+{(\nu+s-1)(d-1)\over d+2\beta+1})\over\Gamma(d+{(\nu-1)(d-1)\over d+2\beta+1})}{\Gamma({d+\nu\over 2}+\beta +1)^d\over\Gamma({d+\nu+s\over 2}+\beta +1)^d}\prod\limits_{i=1}^{d-1}{\Gamma({i+\nu+s+1\over 2})\over \Gamma({i+\nu+1\over 2})}.
\end{align*}
\end{lemma}

\section{Kendall's problem for weighted $\beta$-Delaunay simplices}\label{sec:KendallBeta}

\subsection{Statement and proof of the main theorem}

Consider for $\beta>-1$ the $\beta$-Delaunay tessellation $\cD_\beta$ in $\RR^{d-1}$. For fixed $\nu\ge-1$ we are interested in the shape of a $\nu$-weighed typical cell $Z_{\beta,\nu}$ of $\cD_\beta$, given that $Z_{\beta,\nu}$ has large volume. Previously, this question, also known as Kendall's problem as we explained in the introduction, was addressed for the typical cell of the classical Poisson-Delaunay tessellation by Hug and Schneider \cite{HugSchneiderDelaunay}, who showed that the shape of such cell is close to the shape of a regular simplex. Here, we extend this result in two directions, namely to the case of $\beta$-Delaunay tessellations and at the same time to arbitrary weights $\nu\ge-1$. For a simplex $S\subset\RR^{d-1}$ we denote by $\cF_0(S)$ the set of vertices of $S$. To measure the distance between two simplices $S_1,S_2\subset\RR^{d-1}$ we let
$$
\rho(S_1,S_2):=\inf\{s\geq 0:\text{ for all } v\in\cF_0(S_1)\text{ there exists } w\in\cF_0(S_2)\text{ with }\|v-w\|\leq s\}
$$
be the smallest $s\geq 0$ with the property that for each vertex $v$ of $S_1$ there exists a vertex $w$ of $S_2$ such that their Euclidean distance is bounded by $s$. If $S\subset\RR^{d-1}$ is a simplex, we let $z(S)$ be the centre and $r(S)$ be the radius of its circumsphere, and define
$$
\rho(S) := \inf\{\rho\big(r(S)^{-1}(S-z(S)),T\big):T\text{ a regular simplex with vertices on }\SS^{d-2}\}.
$$
In particular, if the value of $\rho(S)$ is small, we can say that the shape of the simplex $S$ is close to that of a regular simplex.

\begin{theorem}[Kendall's problem for the $\nu$-weighted typical cell]\label{thm:Kendall}
	Fix $d\geq 2$, $\beta>-1$ and $\nu\ge-1$, {such that $d+2\beta+\nu\ge 0$}. Let $\varepsilon\in(0,1)$ and $I=[a,b)$ be an interval, where $\infty\geq b>a\geq a_0$ for some $a_0>0$. Then there exists a constant $c\in(0,\infty)$ only depending on $d$ and $\beta$ and a constant $C\in(0,\infty)$ only depending on $d$, $\beta$, $\varepsilon$, $\nu$ and $a_0$, such that
	$$
	\PP(\rho(Z_{\beta,\nu})\geq\varepsilon\,|\,\Vol(Z_{\beta,\nu})\in I) \leq C\,\exp\Big\{{-c\,\varepsilon^2\,a^{d+1+2\beta\over d-1}}\Big\}.
	$$
	In particular,
	$$
	\lim_{a\to\infty}\PP(\rho(Z_{\beta,\nu})\geq\varepsilon\,|\,\Vol(Z_{\beta,\nu})\geq a) = 0.
	$$
\end{theorem}

We remark that by formally taking $\beta=-1$ and $\nu=0$, where $Z_{\beta,\nu}$ is the typical cell in a classical Poisson-Delaunay tessellation, Theorem \ref{thm:Kendall} reduces to \cite[Theorem 1]{HugSchneiderDelaunay} (with the dimension $d$ there replaced by $d-1$ in our case). The proof of Theorem \ref{thm:Kendall} is based on the explicit representation of the distribution of $Z_{\beta,\nu}$ we developed in part I of this paper and which we recalled in the previous section and otherwise closely follows the principal steps from \cite{HugSchneiderDelaunay}, but with suitable modifications. In particular, the following two more technical results are needed, whose proofs are postponed to Section \ref{subsec:ProofLemmasKendall}. To formulate them, we denote by $\tau_{d-1}$ the $(d-1)$-dimensional volume of a $(d-1)$-dimensional regular simplex with vertices on the unit sphere $\SS^{d-2}$. Moreover, we will use the following stability estimate from \cite[Theorem 2]{HugSchneiderDelaunay}: Let $\varepsilon\in[0,1]$ and $S$ be a simplex with vertices on $\SS^{d-2}$ and with $\rho(S)\geq\varepsilon$. Then there exists a constant $c_1\in(0,\infty)$ depending on $d$ only such that
\begin{equation}\label{eq:StabilityEstimate}
	\Vol(S)\leq (1-c_1\varepsilon^2)\tau_{d-1}.
\end{equation}
Also, recall the definition of the constant $m_{d,\beta}$ from \eqref{eq:DefMdbeta}.

\begin{lemma}\label{lem:Kendall1}
	For all $\varepsilon\in(0,1)$ there is a constant $c_2\in(0,\infty)$, which depends on $d$, $\beta$,  $\varepsilon$ and $\nu$ only, such that for all $h\leq (c_1/(c_1+12))\varepsilon^2=:h_0$ and $a>0$,
	\begin{align*}
		\PP(\Vol(Z_{\beta,\nu})\in a[1,1+h]) &\geq c_2\, h\,a^{\nu+{2d\beta+d^2+1\over d-1}}\\
		&\times\exp\Big\{{-{m_{d,\beta}\tau_{d-1}^{-{d+1+2\beta\over d-1}}}\Big(1+{c_1\over 4}(2^{2(d+\beta)\over d-1}-1)\varepsilon^2\Big)a^{d+1+2\beta\over d-1}}\Big\}.
	\end{align*}
\end{lemma}

\begin{lemma}\label{lem:Kendall2}
	For all $\varepsilon\in(0,1)$ there is a constant $c_3\in(0,\infty)$, which depends on $d$, $\beta$,  $\varepsilon$ and $\nu$ only, such that for all $a>0$ and $h>0$,
	\begin{align*}
		\PP(\Vol(Z_{\beta,\nu})\in a[1,1+h],\rho(Z_{\beta,\nu})\geq\varepsilon)&\leq c_3\, h\,a^{d+1+2\beta\over d-1}\\
		&\times \exp\Big\{-{m_{d,\beta}\tau_{d-1}^{-{d+1+2\beta\over d-1}}}\Big(1+{c_1\over 2}(2^{2(d+\beta)\over d-1}-1)\varepsilon^2\Big)a^{d+1+2\beta\over d-1}\Big\}.
	\end{align*}
\end{lemma}

From now on we use the convention that $c$ denotes a generic constant whose value can change from occasion to occasion. If $c$ depends on a parameter or on several parameters like $d$, $\beta$ etc.\ this is indicated by writing $c(d)$, $c(d,\beta)$ etc. On the other hand we reserve the symbols $c_1,c_2$ and $c_3$ for the constants that appeared in \eqref{eq:StabilityEstimate}, Lemma \ref{lem:Kendall1} and Lemma \ref{lem:Kendall2}, respectively.

\begin{proof}[Proof of Theorem \ref{thm:Kendall}]
	Given Lemma \ref{lem:Kendall1} and Lemma \ref{lem:Kendall2} the proof of Theorem \ref{thm:Kendall} can now be completed similarly as in \cite{HugSchneiderDelaunay}. Let $I:=[a,b)$, $\varepsilon\in(0,1)$ and $a\geq a_0>0$. If $h_0>(b-a)/a$, we put $h:=(b-a)/a$ and observe that $a[1,1+h)=I$. Then Lemma \ref{lem:Kendall1} yields
	\begin{align*}
		\PP(\Vol(Z_{\beta,\nu})\in I) \geq c_2\,ha^{\nu+{2d\beta+d^2+1\over d-1}}\,e^{-Ua^{d+1+2\beta\over d-1}}
	\end{align*}
	with $U:={m_{d,\beta}\tau_{d-1}^{-{d+1+2\beta\over d-1}}}\big(1+{c_1\over 4}(2^{2(d+\beta)\over d-1}-1)\varepsilon^2\big)$. Similarly, Lemma \ref{lem:Kendall2} implies that
	\begin{align*}
		\PP(\Vol(Z_{\beta,\nu})\in I,\rho(Z_{\beta,\nu})\geq\varepsilon) \leq c_3\,ha^{d+1+2\beta\over d-1}e^{-Va^{d+1+2\beta\over d-1}}
	\end{align*}
	with $V:={m_{d,\beta}\tau_{d-1}^{-{d+1+2\beta\over d-1}}}\big(1+{c_1\over 2}(2^{2(d+\beta)\over d-1}-1)\varepsilon^2\big)$. Using now that $a\geq a_0$ we conclude that
	$$
	\PP(\rho(Z_{\beta,\nu})\geq\varepsilon\,|\,\Vol(Z_{\beta,\nu})\in I) \leq c(d,\beta,\nu,\varepsilon,a_0)\,e^{-(V-U)a^{d+1+2\beta\over d-1}} = c(d,\beta,\nu,\varepsilon,a_0)\,e^{-c(d,\beta)\varepsilon^2a^{d+1+2\beta\over d-1}}.
	$$
	If on the other hand $h_0<(b-a)/a$, then $1+h_0\leq b/a$ and $a[1,1+h_0)\subset I$. Using again Lemma \ref{lem:Kendall1} we have that
	$$
	\PP(\Vol(Z_{\beta,\nu})\in I) \geq c(d,\beta,\nu,\varepsilon,a_0)\,h_0a^{d+1+2\beta\over d-1}\,e^{-Ua^{d+1+2\beta\over d-1}}.
	$$
	Next, using Lemma \ref{lem:Kendall2} we obtain, for $i\in\{0,1,2,\ldots\}$,
	\begin{align*}
		&\PP(\Vol(Z_{\beta,\nu})\in a(1+h_0)^i[1,1+h_0],\rho(Z_{\beta,\nu})\geq\varepsilon)\\
		&\leq c(d,\beta,\nu,\varepsilon)h_0((1+h_0)^ia)^{d+1+2\beta\over d-1}e^{-V(a(1+h_0)^i)^{d+1+2\beta\over d-1}}\\
		&=c(d,\beta,\nu,\varepsilon)h_0((1+h_0)^ia)^{d+1+2\beta\over d-1}e^{-(V-U)(a(1+h_0)^i)^{d+1+2\beta\over d-1}}e^{-U(a(1+h_0)^i)^{d+1+2\beta\over d-1}}\\
		&\leq c(d,\beta,\nu,\varepsilon)h_0a^{d+1+2\beta\over d-1}e^{-Ua^{d+1+2\beta\over d-1}}e^{-{1\over 2}(V-U)\,a^{d+1+2\beta\over d-1}}((1+h_0)^i)^{d+1+2\beta\over d-1}e^{-{1\over 2}(V-U)\,(a_0(1+h_0)^i)^{d+1+2\beta\over d-1}}.
	\end{align*}
	Using that $I\subset\bigcup_{i=0}^\infty a(1+h_0)^i[1,1+h_0]$ we arrive at
	\begin{align*}
		\PP(\Vol(Z_{\beta,\nu})\in I,\rho(Z_{\beta,\nu})\geq\varepsilon) &\leq c(d,\beta,\nu,\varepsilon)h_0a^{d+1+2\beta\over d-1}e^{-Ua^{d+1+2\beta\over d-1}}e^{-{1\over 2}(V-U)\,a^{d+1+2\beta\over d-1}}\\
		&\qquad\times\sum_{i=0}^\infty (1+h_0)^{i{d+1+2\beta\over d-1}}e^{-{1\over 2}(V-U)(a_0(1+h_0)^i)^{d+1+2\beta\over d-1}}.
	\end{align*}
	Since the series in the last line converges we find another constant $c(d,\beta,\nu,\varepsilon, a_0)\in(0,\infty)$ such that 
	$$
	\PP(\Vol(Z_{\beta,\nu})\in I,\rho(Z_{\beta,\nu})\geq\varepsilon) \leq c(d,\beta,\nu,\varepsilon,a_0)h_0a^{d+1+2\beta\over d-1}e^{-Ua^{d+1+2\beta\over d-1}}e^{-{1\over 2}(V-U)\,a^{d+1+2\beta\over d-1}}.
	$$
	Thus,
	\begin{align*}
		\PP(\rho(Z_{\beta,\nu})\geq\varepsilon\,|\,\Vol(Z_{\beta,\nu})\in I) &= {\PP(\Vol(Z_{\beta,\nu})\in I,\rho(Z_{\beta,\nu})\geq\varepsilon)\over \PP(\Vol(Z_{\beta,\nu})\in I)}\\
		&\leq c(d,\beta,\nu,\varepsilon,a_0) e^{-{1\over 2}(V-U)\,a^{d+1+2\beta\over d-1}}\\
		&= c(d,\beta,\nu,\varepsilon,a_0)e^{-c(d,\beta)\varepsilon^2a^{d+1+2\beta\over d-1}}.
	\end{align*}
	This completes the argument.
\end{proof}

\subsection{Proof of the technical lemmas}\label{subsec:ProofLemmasKendall}

In this section we are going to prove the two technical results, Lemma \ref{lem:Kendall1} and Lemma \ref{lem:Kendall2}. 

\begin{proof}[Proof of Lemma \ref{lem:Kendall1}]
	Using the representation for $Z_{\beta,\nu}$ in Lemma \ref{theo:typical_cell_stoch_rep} and the substitution $s=m_{\beta,d}r^{d+1+2\beta}$ we have that
	\begin{align*}
		&\PP(\Vol(Z_{\beta,\nu})\in a[1,1+h])\\
		&=\alpha_{\beta,d,\nu}\int_{(\BB^{d-1})^d}\int_0^\infty{\bf 1}(\Delta_{d-1}(ry_1,\ldots,ry_d)\in a[1,1+h])\\
		&\qquad\times r^{2\beta d+d^2+\nu(d-1)}\,e^{-m_{\beta,d}r^{d+1+2\beta}}\,\Delta_{d-1}(y_1,\ldots,y_d)^{\nu+1}\prod_{i=1}^d(1-\|y_i\|^2)^\beta\,\dint r\dint y_1\ldots\dint y_d\\
		&=c(d,\beta,\nu)\int_{(\BB^{d-1})^d}\int_0^\infty{\bf 1}\Big(s^{d-1\over d+1+2\beta}\in {a\,m_{\beta,d}^{d-1\over d+1+2\beta}\over \Delta_{d-1}(y_1,\ldots,y_d)}[1,1+h]\Big)\\
		&\qquad\times s^{2\beta d+d^2-d-2\beta +\nu(d-1)\over d+1+2\beta}\,e^{-s}\,\Delta_{d-1}(y_1,\ldots,y_d)^{\nu+1}\prod_{i=1}^d(1-\|y_i\|^2)^\beta\,\dint s\dint y_1\ldots\dint y_d.
	\end{align*}
	Applying the mean value theorem for integrals, for fixed $y_1,\ldots,y_d\in\BB^{d-1}$ in general position we can find
	$$
	\xi(y_1,\ldots,y_d) \in [\ell,u],
	$$
	where
	$$
	\ell:=m_{d,\beta}\Big({a\over \Delta_{d-1}(y_1,\ldots,y_d)}\Big)^{d+1+2\beta\over d-1}\qquad\text{and}\qquad u:=m_{d,\beta}\Big({a(1+h)\over \Delta_{d-1}(y_1,\ldots,y_d)}\Big)^{d+1+2\beta\over d-1},
	$$
	such that the last expression is equal to
	\begin{align*}
		&c(d,\beta,\nu)a^{d+1+2\beta\over d-1}((1+h)^{d+1+2\beta\over d-1}-1)\int_{(\BB^{d-1})^d}\xi(y_1,\ldots,y_d)^{2d\beta +d^2-d-2\beta+\nu(d-1)\over d+1+2\beta}\,e^{-\xi(y_1,\ldots,y_d)}\\
		&\qquad\times\Delta_{d-1}(y_1,\ldots,y_d)^{\nu+1-{d+1+2\beta\over d-1}}\prod_{i=1}^d(1-\|y_i\|^2)^\beta\,\dint y_1\ldots\dint y_d.
	\end{align*}
	For $\varepsilon\in(0,1)$ define the set
	\begin{align}\label{eq:DefRepsilon}
		R(\varepsilon) := \{(y_1,\ldots,y_d)\in(\BB^{d-1})^d:\Delta_{d-1}(y_1,\ldots,y_d)\geq (1+(c_1/12)\varepsilon^2)^{-1}\tau_{d-1}\}
	\end{align}
	and observe that $R(\varepsilon)$ has positive measure with respect to the $d$-fold product of the Lebesgue measure on $\BB^{d-1}$. Then, taking into account that $d+1+2\beta\ge d-1$, we obtain the lower bound
	\begin{align*}
		&\PP(\Vol(Z_{\beta,\nu})\in a[1,1+h])\geq c(d,\beta,\nu)\,ha^{d+1+2\beta\over d-1}\int_{R(\varepsilon)}I(y_1,\ldots,y_d)\,\dint y_1\ldots\dint y_d,
	\end{align*}
	where the integrand is given by
	\begin{equation}\label{eq:DefIntegrandLemma2}
		\begin{split}
		I(y_1,\ldots,y_d) &:= \xi(y_1,\ldots,y_d)^{2d\beta +d^2-d-2\beta+\nu(d-1)\over d+1+2\beta}\,e^{-\xi(y_1,\ldots,y_d)}\\
		&\qquad\qquad\times\Delta_{d-1}(y_1,\ldots,y_d)^{\nu+1-{d+1+2\beta\over d-1}}\prod_{i=1}^d(1-\|y_i\|^2)^\beta.
		\end{split}
	\end{equation}
	For $(y_1,\ldots,y_d)\in R(\varepsilon)$ in general position we have that
	\begin{align*}
		\xi(y_1,\ldots,y_d) \geq m_{d,\beta}\Big({a\over \Delta_{d-1}(y_1,\ldots,y_d)}\Big)^{d+1+2\beta\over d-1} \geq m_{d,\beta}\Big({a\over \tau_{d-1}}\Big)^{d+1+2\beta\over d-1},
	\end{align*}
	since $\tau_{d-1}$ is the maximal volume of a simplex with vertices in $\BB^{d-1}$ and also
	\begin{align*}
		\xi(y_1,\ldots,y_d) \leq m_{d,\beta}\Big({a(1+h)\over \Delta_{d-1}(y_1,\ldots,y_d)}\Big)^{d+1+2\beta\over d-1}\leq (1+h_0)^{d+1+2\beta\over d-1}m_{d,\beta}\Big({a\,(1+(c_1/12)\varepsilon^2)\over\tau_{d-1}}\Big)^{d+1+2\beta\over d-1},
	\end{align*}
	where $h_0={c_1\over c_1+12}\varepsilon^2$. 
	This implies that
	\begin{align}
		\nonumber\PP(\Vol(Z_{\beta,\nu})\in a[1,1+h]) &\geq c(d,\beta,\nu)\,h\,a^{{d+1+2\beta\over d-1}{2d\beta+d^2-d-2\beta+\nu(d-1)\over d+1+2\beta}+{d+1+2\beta\over d-1}}\\
		\nonumber&\qquad\times\exp\Big\{-(1+h_0)^{d+1+2\beta\over d-1}m_{d,\beta}\Big({a\,(1+(c_1/12)\varepsilon^2)\over\tau_{d-1}}\Big)^{d+1+2\beta\over d-1}\Big\}\\
		\nonumber &\qquad\times\int_{R(\varepsilon)}\Delta_{d-1}(y_1,\ldots,y_d)^{\nu+1-{d+1+2\beta\over d-1}}\prod_{i=1}^d(1-\|y_i\|^2)^\beta\dint y_1\ldots\dint y_d\\
		\nonumber&=c(d,\beta,\nu,\varepsilon)\,h\,a^{\nu+{2d\beta+d^2+1\over d-1}}\\
		&\qquad\times\exp\Big\{-(1+h_0)^{d+1+2\beta\over d-1}m_{d,\beta}\Big({a\,(1+(c_1/12)\varepsilon^2)\over\tau_{d-1}}\Big)^{d+1+2\beta\over d-1}\Big\},\label{eq:ConstLemma2}
	\end{align}
	since the integral is just a constant depending on $d$, $\beta$, $\nu$ and $\varepsilon$. The choice of $h_0$ and $\varepsilon$ implies that
	$$
	\Big((1+h_0)(1+(c_1/12)\varepsilon^2)\Big)^{d+1+2\beta\over d-1} \leq (1+(c_1/4)\varepsilon^2)^{d+1+2\beta\over d-1}\leq (1+(c_1/4)\varepsilon^2)^{\left\lceil {d+1+2\beta\over d-1}\right\rceil}.
	$$
	Since by definition of the constant $c_1$ we have $(c_1/4)\varepsilon^2 <1$ it follows that for any integer $m\geq 0$ we obtain
	$$
	(1+(c_1/4)\varepsilon^2)^m\leq 1+\Big({c_1\over 4}(2^m -1)\Big)\varepsilon^2.
	$$
	Putting $c_2:=c(d,\beta,\nu,\varepsilon)$, this yields the lower bound
	\begin{align*}
		\PP(\Vol(Z_{\beta,\nu})\in a[1,1+h]) &\geq c_2\,ha^{\nu+{2d\beta+d^2+1\over d-1}}\\
		&\qquad\times\exp\Big\{-m_{d,\beta}\tau_{d-1}^{-{d+1+2\beta\over d-1}}\Big(1+{c_1\over 4}(2^{2(d+\beta)\over d-1}-1)\varepsilon^2\Big)a^{d+1+2\beta\over d-1}\Big\},
	\end{align*}
	for all $h\leq h_0$. This completes the argument.
\end{proof}

\begin{proof}[Proof of Lemma \ref{lem:Kendall2}]
	Repeating the same arguments as at the beginning of the proof of Lemma \ref{lem:Kendall1} we see that
	\begin{align*}
		&\PP(\Vol(Z_{\beta,\nu})\in a[1,1+h],\rho((Z_{\beta,\nu}))\geq\varepsilon)\\
		&= c(d,\beta,\nu)ha^{d+1+2\beta\over d-1}\int_{(\BB^{d-1})^d}\xi(y_1,\ldots,y_d)^{2\beta d+d^2-d-2\beta+\nu(d-1)\over d+1+2\beta}\, e^{-\xi(y_1,\ldots,y_d)}\\
		&\qquad\times\Delta_{d-1}(y_1,\ldots,y_d)^{\nu+1-{d+1+2\beta\over d-1}}\,{\bf 1}(\rho(\conv(y_1,\ldots,y_d))\geq\varepsilon)\prod_{i=1}^d(1-\|y_i\|^2)^\beta\,\dint y_1\ldots\dint y_d.
	\end{align*}
	The geometric stability estimate \eqref{eq:StabilityEstimate} says that $\rho(\conv(y_1,\ldots,y_d))\geq\varepsilon$ implies that 
	$$
	\Delta_{d-1}(y_1,\ldots,y_d)\leq(1-c_1\varepsilon^2)\tau_{d-1}.
	$$
	Thus, if ${\bf 1}(\rho(\conv(y_1,\ldots,y_d))\geq\varepsilon)\neq 0$ we have that
	\begin{align*}
		\xi(y_1,\ldots,y_d)  &\geq m_{d,\beta}\Big({a\over \Delta_{d-1}(y_1,\ldots,y_d)}\Big)^{d+1+2\beta\over d-1} \\
		&\geq m_{d,\beta}\Big({a\over (1-c_1\varepsilon^2)\tau_{d-1}}\Big)^{d+1+2\beta\over d-1}\geq m_{d,\beta}\tau_{d-1}^{-{d+1+2\beta\over d-1}}(a(1+c_1\varepsilon^2))^{d+1+2\beta\over d-1}.
	\end{align*}
	It is easy to ensure that for any $\alpha > 0$ and $c\in (0,1)$ there exists a constant $c_4(c,\alpha)$, such that 
	\begin{align}\label{eq:16-12-2019-1}
		x^{\alpha}e^{-x}\leq c_4(c,\alpha)e^{-cx},\qquad x>0.
	\end{align}
	Thus, we have
	\begin{align*}
		\xi(y_1,\ldots,y_d)&^{2\beta d+d^2-d-2\beta+\nu(d-1)\over d+1+2\beta}\, e^{-\xi(y_1,\ldots,y_d)}\\
		&\leq c(d,\beta,\nu, \varepsilon)\exp\Big\{(1-(c_1/(2c_1+1))\varepsilon^2)^{d+1+2\beta\over d-1}\xi(y_1,\ldots,y_d)\Big\}\\
		&\leq c(d,\beta,\nu, \varepsilon)\exp\Big\{-m_{d,\beta}\tau_{d-1}^{-{d+1+2\beta\over d-1}}a^{d+1+2\beta\over d-1}\Big((1+c_1\varepsilon^2)(1-(c_1/(2c_1+1))\varepsilon^2)\Big)^{d+1+2\beta\over d-1}\Big\}.
	\end{align*}
	Arguing as in proof of Lemma \ref{lem:Kendall1} we obtain
	$$
	\Big((1+c_1\varepsilon^2)(1-(c_1/(2c_1+1))\varepsilon^2)\Big)^{d+1+2\beta\over d-1} \leq \Big(1+{c_1\over 2}(2^{2(d+\beta)\over d-1}-1)\varepsilon^2\Big).
	$$
	Plugging this into the above expression for $\PP(\Vol(Z_{\beta,\nu})\in a[1,1+h],\rho(Z_{\beta,\nu})\geq\varepsilon)$ proves the claim for suitable constant $c_3\in(0,\infty)$ depending only on the parameters mentioned in the lemma.
\end{proof}

\begin{remark}\label{rem:ConstantC4}
	The estimate \eqref{eq:16-12-2019-1} can be made more precise. In fact, we start by writing
	$$
	x^\alpha e^{-x} = x^\alpha e^{-cx}e^{-(1-c)x} \leq e^{-cx}\sup_{x>0}[x^{\alpha}e^{-(1-c)x}].
	$$
	The supremum is attained at $x=\alpha/(1-c)$ so that $c_4(c,\alpha)$ can be chosen as
	$$
	c_4(c,\alpha) = \Big({\alpha\over 1-c}\Big)^\alpha e^{-\alpha}.
	$$
	In particular, if $\delta>0$, $p\geq 1$ and $c=(1-\delta\varepsilon^2)^p$ we have that $
	c_4(c,\alpha) \leq c(\alpha,\delta,p)\varepsilon^{-2\alpha}$ for a suitable constant $c(\alpha,\delta,p)\in(0,\infty)$. We will take advantage of this bound in the proof of Theorem \ref{thm:VolumeTails} below, where we apply this bound with $\alpha={2\beta d+d^2-d-2\beta+\nu(d-1)\over d+1+2\beta}$.
\end{remark}

\subsection{Results about tail asymptotics}

In this section we are interested in the behaviour of the tails of the volume of the $\nu$-weighted typical cell of a $\beta$-Delaunay tessellations at zero and infinity. We will prove that $\PP(\Vol(Z_{\beta,\nu})\leq a)$ behaves like a power of $a$, as $a\to 0$, and that $\PP(\Vol(Z_{\beta,\nu})\leq a)$ decays exponentially, as $a\to\infty$. To state our result we recall that $m_{d,\beta}$ is defined at \eqref{eq:DefMdbeta} and that $\tau_{d-1}$ stands for the volume of a regular simplex with vertices on the $(d-2)$-dimensional sphere $\SS^{d-2}$.

\begin{theorem}[Tail asymptotics for the volume of the $\nu$-weighted typical cell]\label{thm:VolumeTails}
	Fix $d\geq 2$, $\beta>-1$ and $\nu\ge-1$.  
	\begin{itemize}
		\item[(i)] We have that $\lim\limits_{a\to\infty}a^{-{d+1+2\beta\over d-1}}\log\PP(\Vol(Z_{\beta,\nu})\geq a) = -m_{d,\beta}\tau_{d-1}^{-{d+1+2\beta\over d-1}}$ for $d+2\beta+\nu\ge 0$.
		
		\item[(ii)] {Suppose that $d\ge 2(1-\beta)$. Then we have that $\lim\limits_{a\to 0}a^{-(\nu+2)}\,\PP(\Vol(Z_{\beta,\nu})\leq a) = C$ for some constant $C\in(0,\infty)$ only depending on $d$, $\beta$ and $\nu$.}
		\end{itemize}
		\end{theorem}

{\begin{remark}
\begin{itemize}
\item[(i)]  We remark that the result of Theorem \ref{thm:VolumeTails} (i) reduces to a special case of \cite[Theorem 2]{HSDelaunay2} when we formally put $\beta=-1$, which corresponds to the classical Poisson-Delaunay tessellation. On the other hand, the result of part (ii) of Theorem \ref{thm:VolumeTails} is new even for this classical model.
\item[(ii)] The dimension restriction $d\ge 2(1-\beta)$ in Theorem \ref{thm:VolumeTails} (ii) is in fact restrictive only for $d\in\{2,3\}$. Namely, if $d=2$, it requires that $\beta\geq 0$, for $d=3$ that $\beta\geq -1/2$, while for $d\geq 4$ it is automatically fulfilled, since $\beta>-1$.
\end{itemize}
\end{remark}}

\begin{proof}[Proof of Theorem \ref{thm:VolumeTails}]
	We start with part (i). Using Lemma \ref{lem:Kendall1} with $h=h_0$ gives 
	\begin{align*}
		\PP(\Vol(Z_{\beta,\nu})\geq a) \geq c(d,\beta,\varepsilon,\nu)a^{\nu+{2d\beta+d^2+1\over d-1}}\exp\Big\{-m_{d,\beta}\tau_{d-1}^{-{d+1+2\beta\over d-1}}(1+c(d,\beta)\varepsilon^2)a^{d+1+2\beta\over d-1}\Big\}.
	\end{align*}
	From now on we shall take $\varepsilon=1/a$.
	
	At this place, we need to make more precise the dependence of the constant $c(d,\beta,\varepsilon,\nu)$ on $\varepsilon$. This dependence arises precisely at \eqref{eq:ConstLemma2} and an inspection of the proof there shows that $c(d,\beta,\varepsilon,\nu)$ can be chosen as $\varepsilon^{d-1}c(\nu,d,\beta)$. This is because the the integral over the region $R(\varepsilon)$ defined at \eqref{eq:DefRepsilon} of the function in \eqref{eq:DefIntegrandLemma2} can be bounded from below by a constant multiple $c(d,\beta,\nu)$ of the volume of $R(\varepsilon)$, which in turn is bounded from below by a constant multiple of $\varepsilon^{d-1}$.	Thus,
	\begin{align*}
		\liminf_{a\to\infty}a^{-{d+1+2\beta\over d-1}}\log\PP(\Vol(Z_{\beta,\nu})\geq a) &\geq -m_{d,\beta}\tau_{d-1}^{-{d+1+2\beta\over d-1}}+\liminf_{a\to\infty}a^{-{d+1+2\beta\over d-1}}\log(\varepsilon^{d-1}c(\nu,d,\beta)).
	\end{align*}
	Now, since we took $\varepsilon=1/a$, the last limit vanishes and we arrive at
	\begin{align*}
		\liminf_{a\to\infty}a^{-{d+1+2\beta\over d-1}}\log\PP(\Vol(Z_{\beta,\nu})\geq a) \geq -m_{d,\beta}\tau_{d-1}^{-{d+1+2\beta\over d-1}}.
	\end{align*}

	To obtain an upper bound we let $q>0$ be a parameter to be determined later and write
	\begin{align*}
		\PP(\Vol(Z_{\beta,\nu})\geq a) &= \PP\Big(\Vol(Z_{\beta,\nu})\in\bigcup_{i=0}^\infty2^ia[1,2] ,\bigcup_{k=1}^\infty\{\rho(Z_{\beta,\nu})\geq k^{-q}\}\Big)\\
		&\leq \sum_{i=0}^\infty\sum_{k=1}^\infty\PP(\Vol(Z_{\beta,\nu})\in 2^ia[1,2],\rho(Z_{\beta,\nu})\geq k^{-q}).
	\end{align*}
	The probability in the previous expression can be bounded using Lemma \ref{lem:Kendall2} with $h=1$. In fact, we have that
	\begin{align*}
		&\PP(\Vol(Z_{\beta,\nu})\in 2^ia[1,2],\rho(Z_{\beta,\nu})\geq k^{-q})\\
		&\qquad \leq c(d,\beta,\nu)c(k,q)(2^ia)^{d+1+2\beta\over d-1}\exp\Big\{-m_{d,\beta}\tau_{d-1}^{-{d+1+2\beta\over d-1}}(1+c(d,\beta)k^{-2q})(2^ia)^{d+1+2\beta\over d-1}\Big\}\\
		&\qquad \leq c(d,\beta,\nu)c(k,q)(2^ia)^{d+1+2\beta\over d-1}\exp\Big\{-m_{d,\beta}\tau_{d-1}^{-{d+1+2\beta\over d-1}}a^{d+1+2\beta\over d-1}\Big\}\\
		&\qquad\qquad\times \exp\Big\{-m_{d,\beta}\tau_{d-1}^{-{d+1+2\beta\over d-1}}(2^{i-1}a)^{d+1+2\beta\over d-1}\Big\}\exp\Big\{-m_{d,\beta}\tau_{d-1}^{-{d+1+2\beta\over d-1}}c(d,\beta)k^{-2q}a^{d+1+2\beta\over d-1}\Big\}
	\end{align*}
	for $i\geq 1$
	Moreover, according to Remark \ref{rem:ConstantC4} the constant $c(k,q)$ satisfies
	$$
	c(k,q) \leq c(d,\beta,\nu)k^{-2q{2\beta d+d^2-d-2\beta+\nu(d-1)\over d+1+2\beta}}
	$$
	and we choose now $q=q(d,\beta,\nu)$ as ${d+1+2\beta\over 2\beta d+d^2-d-2\beta+\nu(d-1)}>0$, which ensures that $c(k,q)\leq c(d,\beta,\nu)k^{-2}$. Then
	\begin{align*}
		\PP(\Vol(Z_{\beta,\nu})\geq a) &\leq c(d,\beta,\nu)a^{d+1+2\beta\over d-1}\exp\Big\{-m_{d,\beta}\tau_{d-1}^{-{d+1+2\beta\over d-1}}a^{d+1+2\beta\over d-1}\Big\}\\
		&\qquad\times\sum_{k=1}^\infty c(k,q)\exp\Big\{-m_{d,\beta}\tau_{d-1}^{-{d+1+2\beta\over d-1}}c(d,\beta)k^{-2q}a^{d+1+2\beta\over d-1}\Big\}\\
		&\qquad{\times\Big(\sum_{i=1}^\infty 2^{i\,{d+1+2\beta\over d-1}}\exp\Big\{-m_{d,\beta}\tau_{d-1}^{-{d+1+2\beta\over d-1}}(2^{i-1}a)^{d+1+2\beta\over d-1}\Big\}+1\Big)}.
	\end{align*}
	Since we are interested in the limit as $a\to\infty$, we may suppose that $a\geq 1$, say, from now on. Then
	the series over $i$ converges to a constant only depending on $d$ and $\beta$. Moreover, since $c(k,q)\leq c(d,\beta,\nu)k^{-2}$ also the series over $k$ converges, but this time to a constant only depending on $d$, $\beta$ and on $\nu$. Thus,
	\begin{align*}
		\PP(\Vol(Z_{\beta,\nu})\geq a) &\leq c(d,\beta,\nu)a^{d+1+2\beta\over d-1}\exp\Big\{-m_{d,\beta}\tau_{d-1}^{-{d+1+2\beta\over d-1}}a^{d+1+2\beta\over d-1}\Big\}
	\end{align*}
	for $a\geq 1$, say, and hence
	\begin{align*}
		\limsup_{a\to\infty}a^{-{d+1+2\beta\over d-1}}\log\PP(\Vol(Z_{\beta,\nu})\geq a) \leq -m_{d,\beta}\tau_{d-1}^{-{d+1+2\beta\over d-1}}.
	\end{align*}
	Together with the lower bound this proves (i).
	
	{In order to show (ii) we use the probabilistic representation of the typical cell $Z_{\beta,\nu}$ from Remark \ref{rem:rep_typical}, which says that $Z_{\beta,\nu}$ has the same distribution as $\conv(RY_1,\ldots, RY_d)$ with the radius $R$ being independent of the beta random simplex $P_{\beta,\nu}:=\conv(Y_1,\ldots, Y_d)$. Writing $\PP_R$ for the distribution of the random variable $R$ we may write, for $a>0$,
	\begin{equation}\label{eq_08.04.21_1}
	\begin{split}
	\PP(\Vol(Z_{\beta,\nu})\leq a)&=\PP(R^{d-1}\Vol(P_{\beta,\nu})\leq a)\\
	&=c_0(d,\beta,\nu)\int_0^\infty\int_{(\BB^{d-1})^d}{\bf 1}(\Delta_{d-1}(y_1,\ldots,y_d)\leq ar^{-d+1})\\
	&\qquad\qquad\times\Delta_{d-1}(y_1,\ldots,y_d)^{\nu+1}\prod_{i=1}^{d}(1-\|y_i\|^2)^{\beta}\,\dint y_1\ldots\dint y_d\,\PP_R(\dint r).
	\end{split}
	\end{equation}
	In the next step we will find the upper bound for the inner $d$-fold integral
	\begin{align*}
	I(b) := \int_{(\BB^{d-1})^d}{\bf 1}(\Delta_{d-1}(y_1,\ldots,y_d)\leq b)\Delta_{d-1}(y_1,\ldots,y_d)^{\nu+1}\prod_{i=1}^{d}(1-\|y_i\|^2)^{\beta}\,\dint y_1\ldots\dint y_d
	\end{align*}
	for general $b>0$. Note that, up to a constant depending on $d$, $\beta$ and $\nu$, $I(b)$ is the same as the probability $\PP(\Vol(P_{\beta,\nu})\leq b)$. We start by observing that
	\begin{align*}
	I(b)&
	\leq b^{\nu+1}\int_{(\BB^{d-1})^d}{\bf 1}(\Delta_{d-1}(y_1,\ldots,y_d)\leq b)\prod_{i=1}^{d}(1-\|y_i\|^2)^{\beta}\,\dint y_1\ldots\dint y_d\\
	&= c_1(d,\beta,\nu)b^{\nu+1}\PP(\Vol(P_{\beta,-1})\leq b),
	\end{align*}
	where $c_1(\beta,\nu,d)$ is some positive constant. The random simplex $P_{\beta,-1}$ is the convex hull of independent beta-distributed random points $Z_1,\ldots, Z_d$ in $\BB^{d-1}$ with density proportional to $(1-\|y\|^2)^\beta$. Let $\dist(L,x):=\min_{y\in L}\|x-y\|$ denote the distance between an affine subspace $L\subset \RR^{d-1}$ and a point $x\in\RR^{d-1}$. Then using the well-known base-times-hight-formula for the volume of a $(d-1)$-dimensional simplex we obtain
	\begin{equation}\label{eq_08.04.21_4}
	\begin{aligned}
	\Vol(P_{\beta,-1})&={1\over (d-1)}\Vol(\conv(Z_1,\ldots,Z_{d-1}))\cdot\dist(\aff(Z_1,\ldots, Z_{d-1}),Z_d)\\
	&=:{1\over (d-1)}\Delta_{d-2}(Z_1,\ldots,Z_{d-1})\cdot \dist(\aff(Z_1,\ldots, Z_{d-1}),Z_d),
	\end{aligned}
	\end{equation}
	where $\aff(z_1,\ldots,z_{d-1})$ denotes the affine hull of points $z_1,\ldots, z_{d-1}\in\RR^{d-1}$. Since $Z_d$ is independent of $Z_1,\ldots, Z_{d-1}$ we have that
	\begin{equation}\label{eq_08.04.21_2}
	\begin{split}
	\PP(\Vol(P_{\beta,-1})\leq b)&=\int_{(\BB^{d-1})^{d-1}}\PP\big(\dist(\aff(y_1,\ldots, y_{d-1}),Z_d)\leq (d-1)b\Delta_{d-2}^{-1}(y_1,\ldots,y_{d-1})\big)\\
	&\hspace{4cm}\PP_{Z_1,\ldots,Z_{d-1}}(\dint(y_1,\ldots,y_{d-1})),
	\end{split}
	\end{equation}
	where $\PP_{Z_1,\ldots,Z_{d-1}}$ stands for the joint distribution of $Z_1,\ldots,Z_{d-1}$.
	For a fixed $(d-2)$-dimensional subspace $L\subset\RR^{d-1}$ and for some $c>0$ we consider the probability $\PP(\dist(L,Z_d)\leq c)$. Let $p(L)\in L$ be the unique point satisfying $\dist(L,0)=\|p(L)\|$. It is clear that
	$
	\dist(L,Z_d)=\|p(L)-\proj_{L^{\perp}}(Z_d)\|,
	$
	where $L^{\perp}$ is the line, passing through $p(L)$ and is orthogonal to $L$, and $\proj_{L^{\perp}}:\RR^{d-1}\mapsto L^{\perp}$ denotes the orthogonal projection onto $L^\perp$. Let $I_{L^{\perp}}:L^{\perp}\mapsto\RR$ be an arbitrary but fixed isometry, such that $I_{L^{\perp}}(0)=0$. As a next step we use \cite[Lemma 3.1]{KTZ20},  which ensures that the random variable $I_{L^{\perp}}(\proj_{L^{\perp}}(Z_d))$ has density 
	$$
	t\mapsto {\Gamma(1/2+d/2+\beta)\over \sqrt{\pi}\Gamma(d/2+\beta)}(1-t^2)^{\beta+{d-2\over 2}}{\bf 1}_{[-1,1]}(t).
	$$
	Hence, since $p(L)\in L^{\perp}$ and the distance is invariant under the isometry $I_{L^{\perp}}$ we obtain
	\begin{align*}
	\PP(\dist(L,Z_d)\leq c)=c_2(d,\beta)\int_{-1}^{1}{\bf 1}(|t-p|\leq c)(1-t^2)^{\beta+{d-2\over 2}}\,\dint t,
	\end{align*}
	where $p:=I_{L^{\perp}}(p(L))$. By the mean value theorem there exists $s\in [\max(p-c,-1), \min(p+c,1)]$ such that
		\begin{equation}\label{eq_10.04.21_2}
	\PP(\dist(L,Z_d)\leq c)\leq 2c_2(d,\beta)(1-s^2)^{\beta+{d-2\over 2}}\cdot c \leq c_3(d,\beta)c.
	\end{equation}
	Substituting this into \eqref{eq_08.04.21_2} and using the explcit description of $\PP_{Z_1,\ldots,Z_{d-1}}$ we get
	\begin{align*}
		\PP(\Vol(P_{\beta,-1})\leq b)&\leq c_4(d,\beta)\,b\int_{(\BB^{d-1})^{d-1}} (\Delta_{d-2}(y_1,\ldots,y_{d-1}))^{-1}\prod_{i=1}^{d-2}(1-\|y_i\|^2)^{\beta}\,\dint y_1\ldots \dint y_{d-1}.
	\end{align*}
	Applying the affine Blaschke-Petkantschin formula \cite[Theorem 7.2.7]{SW} it is easy to see that the last integral is finite and since it is independent of $b$ we conclude 
	$$
	\PP(\Vol(P_{\beta,-1})\leq b)\leq c_5(d,\beta)\,b,
	$$
	for some positive constant $c_5(d,\beta)$. Eventually, this implies that $I(b)\leq c_6(d,\beta,d)b^{\nu+2}$ for any $b>0$. Plugging this back into \eqref{eq_08.04.21_1} and using the definition of $\PP_R$ from Remark \ref{rem:rep_typical} we obtain
	\begin{equation}\label{eq_10.04.21_3}
	\begin{aligned}
\PP(\Vol(Z_{\beta},\nu)\leq a)&\leq c_7(d,\beta,\nu)a^{\nu+2}\int_{0}^{\infty}r^{2d\beta+d^2+\nu(d-1)-(\nu+2)(d-1)}e^{-m_{d,\beta} r^{d+1+2\beta}}\,\dint r\\
&= c_8(d,\beta,\nu)a^{\nu+2}\int_{0}^{\infty}s^{2d\beta-2\beta+d^2-3d+2\over d+1+2\beta}e^{-s}\,\dint s\\
&\leq c_9(d,\beta,\nu)a^{\nu+2},
	\end{aligned}
	\end{equation}
	for some positive constant $c_9(d,\beta,\nu)$ independent of $a$. Note that the last integral is in fact finite for $\beta\geq -3/4$ if $d=2$ and all $\beta>-1$ for $d\geq 3$. Since $d\geq 2(1-\beta)$ by assumption, this is always satisfied.}
	
	{In order to obtain a lower bound we note that by \eqref{eq_08.04.21_1},
	\begin{equation}\label{eq_11.04.21_1}
	\PP(\Vol(Z_{\beta},\nu)\leq a)\ge \int_0^\infty \PP\big(\Vol(P_{\beta,\nu})\leq ar^{-d+1}\big)\,{\bf 1}(r^{d-1}\ge 4(d-1)\tau^{-1}_{d-2}a)\,\PP_R(\dint r),
	\end{equation}
	 where $\tau_{d-2}$ denotes the $(d-2)$-dimensional volume of $(d-2)$-dimensional regular simplex with vertices on the unit sphere $\SS^{d-3}$. 
	In what follows we consider the probability $\PP(\Vol(P_{\beta,\nu})\leq b)$ for $b\leq \tau_{d-2}/(4(d-1))$. For any $C\in[0,1)$ we have
	\begin{align*}
	&\PP(\Vol(P_{\beta,\nu})\leq b)\\
	&\quad\ge c_1(d,\beta,\nu)\int_{(\BB^{d-1})^d}{\bf 1}(Cb< \Delta_{d-1}(y_1,\ldots,y_d)\leq b)\Delta_{d-1}(y_1,\ldots,y_d)^{\nu+1}\prod_{i=1}^{d}(1-\|y_i\|^2)^{\beta}\,\dint y_1\ldots\dint y_d\\
	&\quad\geq c_{10}(d,\beta,\nu)C^{\nu+1}b^{\nu+1}\int_{(\BB^{d-1})^d}{\bf 1}(Cb< \Delta_{d-1}(y_1,\ldots,y_d)\leq b)\prod_{i=1}^{d}(1-\|y_i\|^2)^{\beta}\,\dint y_1\ldots\dint y_d\\
	&\quad= c_{10}(d,\beta,\nu)C^{\nu+1}b^{\nu+1}\Big(\PP(\Vol(P_{\beta,-1})\leq b)-\PP(\Vol(P_{\beta,-1})\leq Cb)\Big).
	\end{align*}
	From \eqref{eq_08.04.21_4} and the fact that $\Delta_{d-2}(z_1,\ldots, z_{d-1})\leq \tau_{d-2}$ for all $z_1,\ldots,z_{d-1}\in\BB^{d-1}$, we conclude together with the base-times-high-formula for the volume of $(d-1)$-dimensional simplices that
	\begin{align*}
	&\PP(\Vol(P_{\beta,-1})\leq b)\\
	&\quad\ge \PP(\dist(\aff(Z_1,\ldots, Z_{d-1}),Z_d)\leq (d-1)\tau_{d-2}^{-1}b)\\
	&\quad\ge \PP\big(\dist(\aff(Z_1,\ldots, Z_{d-1}),Z_d)\leq (d-1)\tau_{d-2}^{-1}b\,|\,\dist(\aff(Z_1,\ldots, Z_{d-1}),0)\leq 1/2\big)\\
	&\quad\qquad\qquad\times\PP(\dist(\aff(Z_1,\ldots, Z_{d-1}),0)\leq 1/2)\\
	&\quad=:c_{11}(d,\beta)\PP\big(\dist(\aff(Z_1,\ldots, Z_{d-1}),Z_d)\leq (d-1)\tau_{d-2}^{-1}b\,|\,\dist(\aff(Z_1,\ldots, Z_{d-1}),0)\leq 1/2\big).
	\end{align*}
	Considering a fixed $(d-2)$-dimensional affine subspace $L\subset\RR^{d-1}$ with $\|p(L)\|\leq 1/2$, where $p(L)$ is defined as before, taking $p:=I_{L^{\perp}}(p(L))$ and applying the mean value theorem we have for any $c\leq 1/4$,
	\begin{align*}
	\PP(\dist(L,Z_d)\leq c)&=c_2(d,\beta)\int_{-1}^{1}{\bf 1}(|t-p|\leq c)(1-t^2)^{\beta+{d-2\over 2}}\,\dint t\\
	&=2c_2(d,\beta)(1-s^2)^{\beta+{d-2\over 2}}\cdot c,	
	\end{align*}
	with $s\in [p-c,p+c]$. Since $s\leq p+c\leq 3/4$ we see that $(1-s^2)\ge 7/16>1/4$, which together with the condition $d\ge 2-2\beta$ leads to $\PP(\dist(L,Z_d)\leq c)\ge c_2(d,\beta)2^{-2\beta-d+3}c$. Hence,
	$$
	\PP(\Vol(P_{\beta,-1})\leq b)\ge c_{11}(d,\beta)c_2(d,\beta)2^{-2\beta-d+3}(d-1)\tau_{d-2}^{-1}b.
	$$
	Finally from \eqref{eq_10.04.21_2} with $c=c_{11}(d,\beta)2^{-2\beta-d+1}(d-1)\tau_{d-2}^{-1}b$ we get
	$$
	\PP\big(\Vol(P_{\beta,-1})\leq c_{11}(d,\beta)2^{-2\beta-d+1}(d-1)\tau_{d-2}^{-1} b\big)\leq c_{11}(d,\beta)c_2(d,\beta)2^{-2\beta-d+2}(d-1)\tau_{d-2}^{-1}b.
	$$
	Thus, taking $C=c_{11}(d,\beta)2^{-2\beta-d+1}(d-1)\tau_{d-2}^{-1}<1$ and combining everything together we obtain
	$$
	\PP(\Vol(P_{\beta,\nu})\leq b)\ge c_{12}(d,\beta,\nu)\,b^{\nu+2}.
$$
Substituting this back into \eqref{eq_11.04.21_1} we conclude for $a\leq 1/2$, that
\begin{align*}
\PP(\Vol(Z_{\beta},\nu)\leq a)&\ge c_{12}(d,\beta,\nu)a^{\nu+2}\int_0^\infty r^{-(\nu+2)(d-2)}{\bf 1}(r^{d-1}\ge 2(d-1)\tau_{d-2})\,\PP_R(\dint r)\\
&=c_{13}(d,\beta,\nu)a^{\nu+2},
\end{align*}
for some positive constant $c_{13}(d,\beta,\nu)$ independent of $a$. Together with \eqref{eq_10.04.21_3} this finishes the proof of part (ii).}
\end{proof}

\section{Log-volume of weighted typical cells in high dimensions}\label{sec:LimitTheorems}

In this section we investigate the asymptotic probabilistic behaviour of the random variables
$$
Y_{\beta,\nu,d}:=\log \Vol(Z_{\beta, \nu}),\qquad \nu\ge -1,\,\, \beta>-1,
$$
as $d+2\beta+\nu\to\infty$, where we use the same notation as before. In particular, we will use the cumulant method, which allows us to obtain a number of asymptotic probabilistic results, such as a central limit theorem with Berry-Esseen bound, a moderate deviation principle and concentration inequalities, as soon as fine estimates for the cumulants of random variable under consideration are established. At the end of the section we will investigate the mod-$\phi$ convergence and the large deviation behaviour based on the exact formulas for the moment generating function of the random variables $Y_{\beta,\nu,d}$.

Given $m\in\NN$ and a random variable $X$ with $\EE|X|^m<\infty$, let $c_m(X)$ be the $m$th cumulant of $X$, which is formally defined as
\[
c_m(X)=(-\mathfrak{i})^m\frac{\dd^m}{\dd t^m}\log \EE[e^{\mathfrak{i}tX}]\big|_{t=0},
\]
where $\mathfrak{i}$ is the imaginary unit. In particular, $c_1(X)=\EE X$ and $c_2(X)=\VV (X)$. Our first step is to derive bounds for the $m$th cumulant and to determine asymptotic formulas for the expectation and the variance of the random variables $Y_{\beta,\nu,d}$. Due to Lemma \ref{lm:VolumeMomBeta} the cumulants of random variables $Y_{\beta,\nu,d}$ can  partially be expressed as a sum of derivatives of gamma functions, also known as polygamma functions. 

In the next subsection we collect some results about the asymptotic behaviour of polygamma functions and some other auxiliary results, which we will need later. Through this section we will use the following notation. Given two functions $f(x)$ and $g(x)$ we write $f=O(g)$ if $\limsup\limits_{x\to\infty}|f(x)/g(x)|<\infty$ and $f=o(g)$ if $\lim\limits_{x\to\infty}|f(x)/g(x)|=0$.

\subsection{Asymptotics for gamma and polygamma functions}\label{sec:polygamma}

The first very useful result regarding the asymptotic behaviour of gamma function is classical Stirling's formula \cite{NIST}:
\begin{align}
\log\Gamma(x)&=x\log x-x-\frac12\log x+{1\over 2}\log(2\pi)+O(1/x),\qquad \text{as}\  x\rightarrow\infty,\, x\in\RR,\label{eq:AsympGamma}\\
\log(n!)&=n\log n-n+\frac12 \log n+O(1),\hspace{2.4cm}\qquad\text{as}\ n\to\infty,\, n\in\mathbb{N}.\label{eq:Stirling}
\end{align}
Consider the digamma function $\psi(x)=\psi^{(0)}(x):=\frac{\dd}{\dd x}\log\Gamma(x)$ and, more generally, the polygamma function 
\[
\psi^{(m)}(x):=\frac{\dd^m}{\dd x^m}\psi(z)=\frac{\dd^{m+1}}{\dd x^{m+1}}\log\Gamma(x),\qquad m\in\mathbb{N}.
\]
The asymptotic expansions for functions $\psi^{(m)}(x)$ are well-known (see e.g. \cite[p. 260]{AS64}), but our estimates will rely on a more precise and recent results for the functions $\psi(x)$ and $\psi^{(1)}(x)$. Namely, in \cite[Theorem C]{QV04} it was shown that
\begin{equation}\label{eq:AsympDigamma}
\psi(x)=\log x-{1\over 2x}+O(1/x^{2})
\end{equation}
and in \cite{Mor10} the asymptotics
\begin{equation}\label{eq:AsympPolygamma}
\psi^{(1)}(x)={1\over x}+{1\over 2x^2}+O(1/x^{3})
\end{equation}
was obtained, as $x\rightarrow\infty$. Moreover, for any $x\neq 0,-1,-2,\ldots$ one has that
$$
\psi^{(m)}(x)=\sum\limits_{k=0}^{\infty}\frac{(-1)^{m+1}m!}{(x+k)^{m+1}},
$$
see \cite[6.4.10]{AS64}. Hence, we conclude that
\begin{equation}\label{eq:PolygammaBound}
|\psi^{(m)}(x)|\leq \frac{(m-1)!}{x^m}+\frac{m!}{x^{m+1}}.
\end{equation}

The following proposition summarizes the results of \cite[Proposition 3.1 - 3.3]{GusakovaThaeleDelaunay} and provides identities or estimates for sums of polygamma functions.

\begin{proposition}[On sums of polygamma functions]
For any $a\in(0,\infty)$ and $k\in\NN$, $k\ge 2$ we have
\begin{align}
\frac12\sum\limits_{j=1}^k\psi\left({j+a\over 2}\right)&=
\left({k-c\over 2}+{a\over 2}-{1\over 2}\right)\psi (a+k-c-1)+{c\over 2}\psi\left(k+a-1\right)+{1\over 4}\psi\left({k+a\over 2}\right)\notag\\
&\qquad-\left({a\over 2}-{1\over 2}\right)\psi(a+1)-{1\over 4}\psi\left({a\over 2}+1\right)-{k\over 2}\left(1+\log 2\right) +1+2c, \label{eq:DigammaSum}
\end{align}
\begin{align}
&\frac14\sum\limits_{j=1}^k\psi^{(1)}\left({j+a\over 2}\right)\notag={1\over 2}\left(\psi(k+a-c+1)-\psi(a+1)\right)+{a\over 2}\left(\psi^{(1)}(k+a-c+1)-\psi^{(1)}(a+1)\right)\notag\\
&\qquad -{1\over 8}\Big(\psi^{(1)}\Big({k+a-c+1\over 2}\Big)-\psi^{(1)}\Big({a+1\over 2}\Big)\Big)+{k-c\over 2}\psi^{(1)}(k+a-c+1)+{c\over 4}\psi^{(1)}\left({k+a\over 2}\right), \label{eq:DerivativeDigammaSum}
\end{align}
and
\begin{align}
\Big|\sum\limits_{j=1}^k\psi^{(m)}\left({j+a\over 2}\right)\Big|&\leq{2^{m+3}m!\over (a+1)^{m-1}}\label{eq:DerivativeDigammaSumBound}.
\end{align}
where $c:= k\mod 2$, which is equal to $0$ if $k$ is even and equal to $1$ if $k$ is odd.
\end{proposition}

\subsection{Cumulant estimates for the log-volume}

In this section we prove general asymptotic formulas and estimates for the cumulants of the random variables $Y_{\beta,\nu,d}$ depending on all three model parameters $\beta$, $\nu$ and $d$.

\begin{proposition}[General cumulant bound and asymptotics for expectation and variance]\label{lem:CumulantBoundsBeta}
For any $\beta>-1$, $\nu\ge -1$, $\gamma >0$ and $d\ge 2$ we have
\begin{align*}
\EE Y_{\beta,\nu,d}&=-{d-1\over 2(d+2\beta+1)}\log (d+2\beta)-{(d-1)(d+2\beta-1)\over 2(d+2\beta+1)}\log(d(d+2\beta+\nu)-(\nu-1))+{d\over 2}\\
&\qquad-{d-1\over 2}\log d+\Big({d+\nu\over 2}-{1\over 4}\Big)\log(d+\nu)-{2\nu+1\over 4}\log(\nu+2)-{d-1\over d+2\beta+1}\log\gamma+O(1),\\
\VV Y_{\beta,\nu,d}&=-{(d-1)^2(d+2\beta-1)\over 2(d(d+2\beta+\nu)-(\nu-1))(d+2\beta +1)}+{d\over 2(d+2\beta+\nu)^2}+{1\over 2}\psi(d+\nu-c+1)\\
&\qquad+{d+\nu-c\over 2}\psi^{(1)}(d+\nu-c+1)-{1\over 8}\psi^{(1)}\Big({d+\nu-c+1\over 2}\Big)+{c\over 4}\psi^{(1)}\left({d+\nu\over 2}\right)\\
&\quad -{1\over 2}\psi(\nu+2)-{\nu+1\over 2}\psi^{(1)}(\nu+2)+{1\over 8}\psi^{(1)}\Big({\nu+2\over 2}\Big)+O((d+2\beta+\nu)^{-2}),
\end{align*}
as $d+2\beta+\nu\rightarrow\infty$, where $c:= (d-1)\mod 2$. Moreover, for any $2\beta+\nu>-d+1$ and $d, m\in\NN$, $m\geq 3$, $d\ge 3$ we have
\begin{align*}
\left|c_m(Y_{\beta,\nu,d})\right|&\leq {(11d+3+6(\beta+2)^{m-1})(m-1)!\over 4(d+2\beta+\nu)^{m-1}}+{4(m-1)!\over (\nu+2)^{m-2}}.
\end{align*}
\end{proposition}

\begin{proof}
By the definition of cumulants we have
$$
c_m(Y_{\beta,\nu,d})=\frac{\dd^m}{\dd s^m}\left[\log \EE \Vol(Z_{\beta,\nu})^{s}\right]\Big|_{s=0}.
$$
Thus, using Lemma \ref{lm:VolumeMomBeta} and the relation $\Gamma(x+1)=x\Gamma(x)$, $x\in(0,\infty)$, we obtain that $c_m(Y_{\beta,\nu,d})$ is equal to
\begin{align*}
&{\bf 1}_{\{m=1\}}\left[{(d-1)\over d+2\beta+1}\Big({1\over 2}\log\pi+ \log \Gamma\Big({d+2\beta+3\over 2}\Big)-\log \Gamma\Big({d+2\beta+2\over 2}\Big)-\log \gamma\Big)-\log (d-1)!\right]\\
&\qquad+\frac{\dd^m}{\dd s^m}\Bigg[\log\Gamma\Big({d(d+2\beta+\nu)\over 2}+{sd\over 2}\Big)+\log\Gamma\Big(d+{(\nu-1)(d-1)\over d+2\beta+1}+{s(d-1)\over d+2\beta +1}\Big)\\
&\qquad-\log\Gamma\Big({d(d+2\beta+\nu)-(\nu-1)\over 2}+{s(d-1)\over 2}\Big)-d\log\Gamma\Big({d+2\beta+\nu\over 2}+{s\over 2}\Big)\\
&\qquad-(d-1)\log(d+2\beta+\nu+s)+\sum\limits_{i=1}^{d-1}\log\Gamma\left({i+\nu+1\over 2}+{s\over 2}\right)\Bigg]\Bigg|_{s=0},
\end{align*}
after simplification of the resulting expression. This expressoin can be re-written in terms of polygamma functions as follows:
\begin{align*}
&{\bf 1}_{\{m=1\}}\left[{(d-1)\over d+2\beta+1}\Big({1\over 2}\log\pi+ \log \Gamma\Big({d+2\beta+3\over 2}\Big)-\log \Gamma\Big({d+2\beta+2\over 2}\Big)-\log\gamma\Big)-\log (d-1)!\right]\\
&\qquad+{d^m\over 2^m}\psi^{(m-1)}\Big({d(d+2\beta+\nu)\over 2}\Big)+{(d-1)^m\over (d+2\beta +1)^m}\psi^{(m-1)}\Big({d(d+2\beta+\nu)-(\nu-1)\over d+2\beta+1}\Big)\\
&\qquad-{(d-1)^m\over 2^m}\psi^{(m-1)}\Big({d(d+2\beta+\nu)-(\nu-1)\over 2}\Big)-{d\over 2^m}\psi^{(m-1)}\Big({d+2\beta+\nu\over 2}\Big)\\
&\qquad+{1\over 2^m}\sum\limits_{i=1}^{d-1}\psi^{(m-1)}\left({i+\nu+1\over 2}\right)-{(-1)^{m+1}(m-1)!(d-1)\over (d+2\beta+\nu)^m}.
\end{align*}
We now distinguish between the three different cases $m=1$, $m=2$ and $m\geq 3$. If $m=1$, $c_1(Y_{\beta,\nu,d})=\EE Y_{\beta,\nu,d}$ and the previous expression simplifies to
\begin{align*}
&{d-1\over d+2\beta+1}\Big(\log \Gamma\Big({d+2\beta+3\over 2}\Big)-\log \Gamma\Big({d+2\beta+2\over 2}\Big)+\psi\Big({d(d+2\beta+\nu)-(\nu-1)\over d+2\beta+1}\Big)-\log\gamma\Big)\\
&\qquad-\log (d-1)!+{d\over 2}\psi\Big({d(d+2\beta+\nu)\over 2}\Big)-{(d-1)\over 2}\psi\Big({d(d+2\beta+\nu)-(\nu-1)\over 2}\Big)\\
&\qquad-{d\over 2}\psi\Big({d+2\beta+\nu\over 2}\Big)+{1\over 2}\sum\limits_{i=1}^{d-1}\psi\left({i+\nu+1\over 2}\right)+O(1),
\end{align*}
where we used the fact that $\beta>-1$. Applying the asymptotic relations \eqref{eq:AsympGamma} and \eqref{eq:AsympDigamma} together with \eqref{eq:DigammaSum}, and taking into account that $d+2\beta+\nu\to\infty$, we obtain that $\EE Y_{\beta,\nu,d}$ equals to
\begin{align*}
&{d-1\over d+2\beta+1}\Big({d+2\beta+2\over 2}\log\Big(1+{1\over d+2\beta+2}\Big)+{1\over 2}\log (d+2\beta+3)+\log(d(d+2\beta+\nu)-(\nu-1))\\
&\quad-\log(d+2\beta+1)\Big)-\log d!+\log d+{d\over 2}\log d -{(d-1)\over 2}\log (d(d+2\beta+\nu)\\
&\quad-(\nu-1))-{\nu\over 2}\log(\nu+2)-{1\over 4}\log(\nu+3)+{d+\nu-1-c\over 2}\log(\nu+d-c-1)+{c\over 2}\log\left(\nu+d-1\right)\\
&\quad+{1\over 4}\log(\nu+d)-{d\over 2}-{d-1\over d+2\beta+1}\log\gamma+O(1).
\end{align*}
Using the Taylor expression for the logarithm with the Lagrange form of the remainder, namely,
$$
\log(x+a)=\log(x)+{a\over x}+O(a^2/x^2),
$$ 
and Stirling's formula  \eqref{eq:Stirling} we have
\begin{align*}
\EE Y_{\beta,\nu,d}&=-{d-1\over 2(d+2\beta+1)}\log (d+2\beta)-{(d-1)(d+2\beta-1)\over 2(d+2\beta+1)}\log(d(d+2\beta+\nu)-(\nu-1))\\
&\quad-{d-1\over 2}\log d+\Big({d+\nu\over 2}-{1\over 4}\Big)\log(d+\nu)-{2\nu+1\over 4}\log(\nu+2)+{d\over 2}-{d-1\over d+2\beta+1}\log\gamma+O(1).
\end{align*}
Next, we turn to the case $m=2$. We have that $c_2(Y_{\mu, n})=\VV Y_{\mu,n}$ equals to
\begin{align*}
&{d^2\over 4}\psi^{(1)}\Big({d(d+2\beta+\nu)\over 2}\Big)+{(d-1)^2\over (d+2\beta +1)^2}\psi^{(1)}\Big({d(d+2\beta+\nu)-(\nu-1)\over d+2\beta+1}\Big)-{d\over 4}\psi^{(1)}\Big({d+2\beta+\nu\over 2}\Big)\\
&\qquad-{(d-1)^2\over 4}\psi^{(1)}\Big({d(d+2\beta+\nu)-(\nu-1)\over 2}\Big)+{1\over 4}\sum\limits_{i=1}^{d-1}\psi^{(1)}\left({i+\nu+1\over 2}\right)+{d-1\over (d+2\beta+\nu)^2}.
\end{align*}
Taking into account that $d+2\beta+\nu\to\infty$ and using \eqref{eq:AsympPolygamma} and \eqref{eq:DerivativeDigammaSum} together with a Taylor expansion of the function $1/x$ with the Lagrange form of the remainder we conclude that
\begin{align*}
\VV Y_{\beta,\nu,d}&=-{(d-1)^2(d+2\beta-1)\over 2(d(d+2\beta+\nu)-(\nu-1))(d+2\beta +1)}+{d\over 2(d+2\beta+\nu)^2}+{1\over 2}\psi(d+\nu-c+1)\\
&\qquad+{d+\nu-c\over 2}\psi^{(1)}(d+\nu-c+1)-{1\over 8}\psi^{(1)}\Big({d+\nu-c+1\over 2}\Big)+{c\over 4}\psi^{(1)}\left({d+\nu\over 2}\right)\\
&\quad -{1\over 2}\psi(\nu+2)-{\nu+1\over 2}\psi^{(1)}(\nu+2)+{1\over 8}\psi^{(1)}\Big({\nu+2\over 2}\Big)+O\Big({1\over (d+2\beta+\nu)^2}\Big).
\end{align*}
This proves the first two assertions of the proposition. 

We turn now to the case that $m\geq 3$. Applying the estimates \eqref{eq:PolygammaBound} and \eqref{eq:DerivativeDigammaSumBound} for $d\ge 3$ and $\nu+2\beta+d>1$, we get
\begin{align*}
\left|c_m(Y_{\beta,\nu,d})\right|&\leq {d^m\over 2^m}\Big|\psi^{(m-1)}\Big({d(d+2\beta+\nu)\over 2}\Big)\Big|+\Big|\psi^{(m-1)}\Big({d(d+2\beta+\nu)-(\nu-1)\over d+2\beta+1}\Big)\Big|\\
&\qquad+{(d-1)^m\over 2^m}\Big|\psi^{(m-1)}\Big({d(d+2\beta+\nu)-(\nu-1)\over 2}\Big)\Big|+{d\over 2^m}\Big|\psi^{(m-1)}\Big({d+2\beta+\nu\over 2}\Big)\Big|\\
&\qquad+\Big|{1\over 2^m}\sum\limits_{i=1}^{d-1}\psi^{(m-1)}\left({i+\nu+1\over 2}\right)\Big|+{(m-1)!(d-1)\over (d+2\beta+\nu)^{m-1}}.
\end{align*}
Since $d+2\beta+\nu>d-2+\nu>\nu-1$ we obtain the upper bound
\begin{align*}
\left|c_m(Y_{\beta,\nu,d})\right|&\leq {d(m-2)!\over (d+2\beta+\nu)^{m-1}}+{2(m-1)!\over (d+2\beta+\nu)^{m}}+{(m-2)!(\beta+2)^{m-1}\over (d+2\beta+\nu)^{m-1}}+{(m-1)!(\beta+2)^{m}\over (d+2\beta+\nu)^{m}}\\
&\qquad+{(d-1)(m-2)!\over 2(d+2\beta+\nu)^{m-1}}+{d(m-1)!\over (d+2\beta+\nu)^{m}}+{(d-1)(m-1)!\over (d+2\beta+\nu)^{m-1}}+{4(m-1)!\over (\nu+2)^{m-2}}\\
&\leq {(11d+3+6(\beta+2)^{m-1})(m-1)!\over 4(d+2\beta+\nu)^{m-1}}+{4(m-1)!\over (\nu+2)^{m-2}}.
\end{align*}
This completes the proof.
\end{proof}

The results of Proposition \ref{lem:CumulantBoundsBeta} allows us to consider the situation, when some of the parameters $d$, $\beta$ and $\nu$ stay fixed and the other(s) tend to infinity with possibly different speed. In particular, assuming that $\beta>-1$ is some fixed number and using the Taylor expansion of the logarithm with the Lagrange form of the remainder we have
\begin{align*}
\log(d+2\beta)&=\log d + O(1/d),\\
\log(d(d+2\beta+\nu)-(\nu-1))&=\log((d-1)(\nu+d-1))+O((d+\nu)^{-1}).
\end{align*}
Substituting these estimates into the expression for the expectation in Proposition \ref{lem:CumulantBoundsBeta} we obtain
\begin{align*}
\EE Y_{\beta,\nu,d}&=-\Big({1\over 2}+O(1/d)\Big)(\log d + O(1/d))-\Big({d-3\over 2}+{2+2\beta\over d+2\beta+1}\Big)\log((d-1)(\nu+d-1))\\
&\qquad-{d-1\over 2}\log d+\Big({d+\nu\over 2}-{1\over 4}\Big)\log(d+\nu)-{2\nu+1\over 4}\log(\nu+2)+{d\over 2}-\log\gamma+O(1)\\
&=-{2d-3\over 2}\log d+\Big({5+2\nu\over 4}+{2+2\beta\over d+2\beta+1}\Big)\log(\nu+d)-{2\nu+1\over 4}\log(\nu+2)+{d\over 2}-\log\gamma+O(1).
\end{align*}
Next, we consider the variance in the same situation. Applying \eqref{eq:AsympDigamma}, \eqref{eq:AsympPolygamma} and the Taylor expansion of the logarithm and of the function $1/x$,
$$
{1\over x+a}={1\over x}-{a\over x^2}+O(a^2/x^3),
$$
we get
\begin{align*}
\VV Y_{\beta,\nu,d}&=-{2d-1\over 4(d+\nu)}+{(\beta+1)d\over (d+\nu)^2}+{(d-1)\over (d+\nu+2\beta+1)(d+2\beta +1)}+{1\over 2}\log(d+\nu)\\
&\qquad+{1\over 2}-{1\over 2}\psi(\nu+2)-{\nu+1\over 2}\psi^{(1)}(\nu+2)+{1\over 8}\psi^{(1)}\Big({\nu+2\over 2}\Big)+O((d+\nu)^{-2}).
\end{align*}
It should be noted, that by Remark \ref{rm:BetaToPoisson} the random variables $Y_{\beta,\nu,d}$ converges weakly to the random variable $Y_{-1,\nu,d}$, as $\beta\to -1$, where $Y_{-1,\nu,d}$ is a logarithmic volume of the $\nu$-weighted typical cell in the classical Poisson-Delaunay tessellation in $\RR^{d-1}$ with intensity ${2\gamma\pi^{d/2}\over \Gamma({d\over 2})}$. This case was considered in \cite{GusakovaThaeleDelaunay} and the case of fixed $\beta$ can be treated analogously.

In this article we focus from now on on the particularly interesting situation when $\nu$ is some fixed number. This covers, for example, the case of the typical ($\nu=0$) cell and the zero-cell ($\nu=1$) of the $\beta$-Delaunay tessellation. For simplicity we let $\beta$ be fixed as well.

\begin{corollary}[Cumulant bound, expectation and variance asymptotics for fixed model parameters]\label{cor:CumulantsTypicalBeta}
Fix $\nu\ge -1, \beta>-1$ and $\gamma>0$. Then, as $d\to\infty$,
\begin{align*}
\EE Y_{\beta,\nu,d}&=-d\log d+{11+2\nu\over 4}\log d+{d\over 2}-\log\gamma+O(1),\\
\VV Y_{\beta,\nu,d}&={1\over 2}\log d+C_{\nu}+O(1/d),
\end{align*}
where 
$$
C_{\nu}:=-{1\over 2}\psi(\nu+2)-{\nu+1\over 2}\psi^{(1)}(\nu+2)+{1\over 8}\psi^{(1)}\Big({\nu+2\over 2}\Big),
$$
and the hidden constant in $O$-big notation only depends on $\nu$ and $\beta$. Moreover, for $d>\max(4,1-2\beta-\nu)$ and $m\ge 3$ we have that
$$
|c_m(Y_{\beta,\nu,d})|\leq {(16+2(\beta+2)^{m-1})(m-1)!\over (\nu+2)^{m-2}}\leq {18(\beta+2)^{m-1}(m-1)!\over (\nu+2)^{m-2}}.
$$
\end{corollary}

\subsection{Central limit theorem and moderate deviations}

As was already mentioned above, based on cumulant bounds for the random variables $Y_{\beta,\nu,d}$ we can prove a number of probabilistic limit theorems for the logarithmic volume of the $\nu$-weighted typical cell $Z_{\beta,\nu}$ of the $\beta$-Delaunay tessellation. 

For completeness, we recall the definition of a large and moderate deviation principle for a sequence of random variables. Given a sequence $(\mu_d)_{d\in\mathbb{N}}$ of probability measures on a topological space $E$, we say that it fulfils a large deviation principle with speed $a_d$ and (good) rate function $I:E\rightarrow [0;\infty]$, if $I$ is lower semi-continuous and has compact level sets, and if for every Borel set $B\subseteq E$ we have
\[
-\inf_{x\in \inter(B)} I(x)\leq \liminf_{d\rightarrow \infty}a_d^{-1}\log\mu_d(B)\leq \limsup_{d\rightarrow\infty}a_d^{-1}\log \mu_d(B)\leq -\inf_{x\in \cl(B)} I(x),
\]
where $\inter(B)$ and $\cl(B)$ stand for the interior and the closure of $B$, respectively. We say that a sequence $(X_d)_{d\in\mathbb{N}}$ of random variables satisfies a large deviations principle if the sequence of their distributions does. Moreover, if the rescaling $a_d$ lies between that of a law of large numbers and that of a distributional (often a central) limit theorem, we will say that a sequence $(X_d)_{d\in\mathbb{N}}$ satisfies a moderate deviations principle with speed $a_d$ and rate function $I$, see \cite{DZ}.

We consider the centred and normalized random random variables
$$
\widetilde Y_{\beta, \nu, d}:=\frac{Y_{\beta, \nu, d}-\EE Y_{\beta, \nu, d}}{\sqrt{\VV Y_{\beta, \nu, d}}}
$$
and define
$$
\varepsilon_{\beta,\nu,d}:={2(\beta+2)\over (\nu+2)\sqrt{\log d}}.
$$
In what follows we denote by $\Phi(\,\cdot\,)$ the distribution function of a standard Gaussian random variable.

\begin{theorem}[Berry-Esseen bound and moderate deviations for the log-volume]\label{thm:Main}
	Suppose that $\nu\ge -1$ and $\beta>-1$ are some fixed real numbers. Then the following assertions hold.
	\begin{itemize}
			\item[(i)] There exists a constant $c\in(0,\infty)$ such that for all sufficiently large $d$,
		$$
		\sup_{y\in\RR}|\PP(\widetilde{Y}_{\beta,\nu,d}\leq y)-\Phi(y)| \leq c\,\varepsilon_{\beta,\nu,d}.
		$$
		
		\item[(ii)] Let $(a_d)_{d\in\NN}$ be a sequences of positive real numbers such that
		$$
		\lim_{d\to\infty}a_d=\infty\qquad\text{and}\qquad \lim_{d\to\infty}a_d\,\varepsilon_{\beta,\nu,d}= 0.
		$$
		Then the sequence of random variables $(a_d^{-1}\widetilde{Y}_{\beta,\nu,d})_{d\in\NN}$ satisfies a moderate deviations principle on $\RR$ with speed $a_d^2$ and rate function $I(x)=x^2/2$.
			\end{itemize}
\end{theorem}

The proof of Theorem \ref{thm:Main} relies on the following lemma, which summarizes some of the main finings of the method of cumulants for normal approximation. We refer the reader to the recent survey article \cite{DJSSurvey} and the many references provided therein. 

 \begin{lemma}\label{lem:ConsequenceCumBounds}
 	Let $(X_d)_{d\in\NN}$ be a sequence of random variables with $\EE[X_d] = 0$ and $\VV[X_d] = 1$ for all $d\in\NN$. Suppose that, for all $m\in \NN$, $m\ge 3$ and sufficiently large $n$,
 	\begin{align}\label{eq:CumBoundAbstract}
 	|c_m(X_d)| \leq \frac{(m!)^{1 + \delta}}{(\Delta_d)^{m-2}}
 	\end{align}
 	with a constant $\delta \in [0,\infty)$ not depending on $d$ and constants $\Delta_d \in (0,\infty)$ that may depend on $d$.  Then the following assertions are true.
 	\begin{itemize}
 		\item[(i)] One has the Berry-Esseen bound
\begin{align*}
	\sup\limits_{y\in \mathbb{R}} |\PP(X_d \leq y) - \Phi(y)| \leq c\, (\Delta_d)^{-1/(1+2\delta)}
\end{align*}
with a constant $c\in (0,\infty)$ that only depends on $\delta$.
 		\item[(ii)] Let $(a_d)_{d \in\NN}$ be a sequence of positive real numbers such that
 		\begin{align*}
 		\lim\limits_{d \rightarrow \infty} a_d = \infty \qquad \text{and}\qquad \lim\limits_{d \rightarrow \infty} a_d\, \Delta_d^{-\frac{1}{1 + 2\delta}} = 0.
 		\end{align*}
 		Then $(a_d^{-1} X_d)_{d\in\NN}$ satisfies a moderate deviations principle on $\mathbb{R}$ with speed  $a_d^2$ and rate function $I(x) = \frac{x^2}{2}$.
 	\end{itemize} 
 \end{lemma}

\begin{proof}[Proof of Theorem \ref{thm:Main}]
	From the estimates in Corollary \ref{cor:CumulantsTypicalBeta} we obtain
	\[
	\big|c_m\big(\widetilde Y_{\beta, \nu, d}\big)\big|=\frac{\left|c_m\left(Y_{\beta, \nu, d}\right)\right|}{\left(\VV Y_{\beta, \nu, d}\right)^{m/2}}
	\leq {36(\beta+2)\over m\sqrt{\log d}}{m!}\varepsilon_{\beta,\nu,d}^{m-2}
	\leq{m!}\,\varepsilon_{\beta,\nu,d}^{m-2},
	\]
	for sufficiently large $d$. Thus, $\widetilde Y_{\beta,\nu,d}$ satisfies condition \eqref{eq:CumBoundAbstract} with $\delta=0$ and $\Delta_d=\varepsilon_{\beta,\nu,d}^{-1}$ respectively. An application of Lemma \ref{lem:ConsequenceCumBounds} now completes the proof.
\end{proof}

\subsection{Mod-Gaussian convergence}

In this subsection we investigate mod-$\phi$ convergence of the logarithmic volume $Y_{\beta,\nu,d}$ of the random simplex $Z_{\beta,\nu}$, assuming, as at the end of the previous section, that the parameters $\beta$ and $\nu$ are fixed and the dimension $d$ tends to infinity. The notion of mod-$\phi$ convergence has been introduced and studied in the previous decade in \cite{DKN15, JKN11}. It is a powerful tool which leads to a whole collection of limit theorems including an extended version of the central limit theorem, a local limit theorem, precise moderate and large deviations and Cram\'er-Petrov type asymptotic expansions. For more references and a survey of the topic we refer the reader to \cite{ModPhiBook}. We remark that \textit{some} of the results established in the previous section will also follow once we have established mod-$\phi$ convergence.

The main idea behind the concept of mod-$\phi$ convergence of a sequence of random variables is to look for a suitable renormalization of their moment generating functions (considered on the complex plane $\mathbb{C}$). There are a several versions and we will consider the one from \cite[Definition 1.1]{ModPhiBook}. Let $(X_d)_{d\in\mathbb{N}}$ be a sequence of real-valued random variables, and let us denote by $\varphi_d(z)=\EE[e^{zX_d}]$ their moment generating functions, which are assumed to exist in a strip 
\[
S_{(a,b)}:=\left\{z\in\CC\colon a<\Re z<b\right\},
\]
where $a<0<b$ are extended real numbers. Assume, that there exists a non-constant infinitely divisible distribution $\phi$ with moment generating function $\int_{\RR}e^{zx}\phi(\dd x)=\exp(\eta(z))$, which is well defined on $S_{(a,b)}$, and an analytic function $\psi$ which does not vanish on the real part of $S_{(a,b)}$, such that 
\[
\exp\left(-w_d\eta(z)\right)\varphi_d(z)\rightarrow \psi(z),\qquad d\to\infty,
\]
locally uniformly in $z\in S_{(a,b)}$ for some sequence $w_d\to\infty$. Then we say that the sequence $(X_d)_{d\in\mathbb{N}}$ converges in the mod-$\phi$ sense on $S_{(a,b)}$ with parameters $(w_d)_{d\in\mathbb{N}}$ and limiting function $\psi$. In particular, if $\eta(z)=z^2/2$, one speaks about mod-Gaussian convergence. This will be the case for our application presented below.

Mod-$\phi$ convergence for the log-volume of different models of random simplices was recently studied in \cite{EKGammaMoments,GKT17, GusakovaThaeleDelaunay}. We remark that although \cite{EKGammaMoments} studies very general models with so-called gamma type moments, our random variables do not precisely fit into this framework. Our argument closely follows the one in \cite{GusakovaThaeleDelaunay} with suitable modifications and adaptions, of course.

Before we state the main result of this section let us recall the definition of the Barnes $G$-function. The Barnes $G$-function is an entire function of one complex argument $z\in\CC$, which can be defined as a solution of the functional equation 
\[
G(z+1)=\Gamma(z)G(z),
\]
satisfying the `initial' condition $G(1)=1$. 

\begin{theorem}[Mod-Gaussian convergence of the log-volume]\label{thm:ModPhi}
For fixed $\nu\ge -1$ and $\beta>-1$ define the sequence $(m_d)_{d\in\NN}$ by
\[
m_d:=\Big({9\over 4}+{\nu\over 2}\Big)\log\Big({d\over 2}\Big)-{d+\nu+2\beta\over 2}-1+\log\Big({\pi\over 4\gamma}\Big)-{3(\beta+1)\log d\over d+2\beta+1}-\log(d-1)!.
\] 
Then, as $d\to\infty$, the sequence of random variables $(Y_{\beta, \nu,d}-m_d)_{d\in\NN}$ converges in the mod-Gaussian sense on the strip $S_{(-\nu-1,\infty)}$ with parameters $w_d={1\over 2}\log\left({d\over 2}\right)-1$ and limiting function 
$$
\psi(z)={G\left({\nu+2\over 2}\right)G\left({\nu+3\over 2}\right)\over G\left({\nu+2+z\over 2}\right) G\left({\nu+3+z\over 2}\right)}.
$$
\end{theorem}

\begin{proof}
Consider the moment generating function 
$$
\varphi_d(z):=\EE[\exp(zY_{\beta,\nu,d})]=\EE[\Vol(Z_{\beta,\nu})^z],\qquad z\in\CC,
$$
of the random variable $Y_{\beta,\nu,d}$.
Then by Lemma \ref{lm:VolumeMomBeta} we obtain the representation
$$
\log\varphi_d(z):=z \left[{d-1\over d+2\beta+1}\log\Big({ \sqrt{\pi}\Gamma({d+1\over 2}+\beta+1)\over \gamma\Gamma({d\over 2}+\beta+1)}\Big)-\log(d-1)!\right]+S_d(z)+T_d(z),
$$
where
\begin{align*}
S_d(z):&=\log \prod\limits_{i=1}^{d-1}{\Gamma({i+\nu+s+1\over 2})\over \Gamma({i+\nu+1\over 2})},\\
T_d(z):&=-\left(\log\Gamma\Big({d(d+2\beta)+\nu(d-1)+1\over 2}+{z(d-1)\over 2}\Big)-\log\Gamma\Big({d(d+2\beta)+\nu(d-1)+1\over 2}\Big)\right)\\
&+\left(\log\Gamma\Big({d(d+\nu +2\beta)\over 2}+1+{zd\over 2}\Big)-\log\Gamma\Big({d(d+\nu +2\beta)\over 2}+1\Big)\right)\\
&+\left(\log\Gamma\Big(d+{(\nu-1)(d-1)\over d+2\beta+1}+{z(d-1)\over d+2\beta+1}\Big)-\log\Gamma\Big(d+{(\nu-1)(d-1)\over d+2\beta+1}\Big)\right)\\
&-d\left(\log\Gamma\Big({d+\nu+2\beta\over 2} +1+{z\over 2}\Big)-\log\Gamma\Big({d+\nu+2\beta\over 2} +1\Big)\right).
\end{align*}
The first term $S_d(z)$ is independent of $\beta$ and was analysed already in \cite{GusakovaThaeleDelaunay}. In fact, from the proof of Theorem 5.2 in \cite{GusakovaThaeleDelaunay} we deduce that
$$
S_d(z)=\log\psi(z)+{z\over 2}\Big((\nu+d-{1\over 2})\log\Big({d-1\over 2}\Big)-d+1+\log(2\pi)\Big)+{z^2\over 4}\log\Big({d-1\over 2}\Big)+O\Big({|z|^3+1\over d}\Big),
$$
and using the Taylor expansion of the logarithm we conclude
\begin{equation}\label{eq:20.01.21_1}
S_d(z)=\log\psi(z)+{z\over 2}\Big((\nu+d-{1\over 2})\log\Big({d\over 2}\Big)-d+\log(2\pi)\Big)+{z^2\over 4}\log\Big({d\over 2}\Big)+O\Big({|z|^3+1\over d}\Big).
\end{equation}

Let us now consider the second term $T_d(z)$. We will use the classical Binet's formula for the logarithm of Gamma function \cite[page 243]{WW}:
$$
\log\Gamma(z)=\Big(z-{1\over 2}\Big)\log z-z+{1\over 2}\log(2\pi)+\int_{0}^{\infty}{e^{-tz}\over t}\Big({1\over 2}-{1\over t}+{1\over e^t-1}\Big)\dint t,\qquad \Re z\in (0,\infty),
$$
in particular
$$
\log\Gamma(z+y)-\log\Gamma(z)=\Big(z-{1\over 2}\Big)\log \Big(1+{y\over z}\Big)+y\log(z+y)-y+\int_{0}^{\infty}{e^{-tz}(e^{-ty}-1)\over t}\Big({1\over 2}-{1\over t}+{1\over e^t-1}\Big)\dint t.
$$
Repeated application of this identity shows that $T_d(z)$ equals
\begin{align*}
&-{(d-1)(d+\nu+2\beta+1)+1+2\beta\over 2}\log\Big(1+{z(d-1)\over (d-1)(d+\nu+2\beta+1)+2+2\beta}\Big)\\
&-{z(d-1)\over 2}\log((d-1)(d+\nu+z+2\beta+1)+2+2\beta)+{z(d-1)\over 2}(1+\log 2)\\
&+{d(d+\nu +2\beta)+1\over 2}\log\Big(1+{zd\over d(d+\nu+2\beta)+2}\Big)+{zd\over 2}\log(d(d+\nu +2\beta+z)+2)-{zd\over 2}(1+\log 2)\\
&+\Big({(d-1)(\nu+d+1+2\beta)+2\beta+2\over d+2\beta+1}-{1\over 2}\Big)\log\Big(1+{z(d-1)\over (d-1)(\nu+d+1+2\beta)+2\beta+2}\Big)\\
&+{z(d-1)\over d+2\beta+1}\log\Big({(d-1)(\nu+d+1+2\beta+z)+2\beta+2\over d+2\beta+1}\Big)-{z(d-1)\over d+2\beta+1}\\
&-{d(d+\nu+2\beta+1)\over 2}\log\Big(1+{z\over d+\nu+2\beta+2}\Big)-{zd\over 2}\log(d+\nu+2\beta+z+2)+{zd\over 2}(1+\log 2)+R_d(z),
\end{align*}
where
\begin{align*}
R_d(z)&:=\int_{0}^{\infty}{1\over t}\Big({1\over 2}-{1\over t}+{1\over e^t-1}\Big)\Big(e^{-{(d(d+\nu +2\beta)+2)t\over 2}}(e^{-{zdt\over 2}}-1)+e^{-{((d-1)(\nu+d+1+2\beta)+2\beta+2)t\over d+2\beta+1}}(e^{-{z(d-1)t\over d+2\beta+1}}-1)\\
&\qquad -e^{-{((d-1)(d+\nu+2\beta+1)+2+2\beta)t\over 2}}(e^{-{z(d-1)t\over 2}}-1)-de^{-{(d+\nu+2\beta+2)t\over 2}}(e^{-{zt\over 2}}-1)\Big)\dint t.
\end{align*}
Now, we note, that the function $t\mapsto {1\over t}({1\over 2}-{1\over t}+{1\over e^t-1})$ is bounded by $1/12$ and that the inequality $|e^{-z}-1|\leq e^{-\Re z}+1$ holds for all $z\in\CC$. By the triangle inequality this leads to
\begin{align*}
|R_d(z)|&\leq{1\over 12}\int_{0}^{\infty}\Big(e^{{-(d\Re z+d(d+\nu +2\beta)+2)t\over 2}}+e^{{-(d(d+\nu +2\beta)+2)t\over 2}}+e^{{-((d-1)\Re z+(d-1)(\nu+d+1+2\beta)+2\beta+2)t\over d+2\beta+1}}\\
&\qquad+e^{{-((d-1)(\nu+d+1+2\beta)+2\beta+2)t\over d+2\beta+1}} +e^{-((d-1)\Re z+(d-1)(d+\nu+2\beta+1)+2+2\beta)t\over 2}+e^{-((d-1)(d+\nu+2\beta+1)+2+2\beta)t\over 2}\\
&\qquad+e^{-(\Re z+d+\nu+2\beta+2)t\over 2}+e^{-(d+\nu+2\beta+2)t\over 2}\Big)\dint t.
\end{align*}
Since $z\in S_{(-\nu-1,\infty)}$ we have that $\Re z>-\nu-1$, showing that all exponents in the integral above are negative. Thus,
$
R_d(z)=O\Big({1\over d}\Big).
$
Further, using the Taylor expansion for the complex logarithm 
$
\log(1+z)=z-{z^2\over 2}+O(|z|^3)
$
we simplify $T_d(z)$ as follows:
\begin{align*}
T_d(z)&={z\over 2}\Big((d-1)\log 2+{4(\beta+1)\over d+2\beta+1}\Big)+{z^2\over 4}-{z(d-1)\over 2}(\log(d-1)+\log(d+\nu+z+2\beta+1))\\
&\qquad+{zd\over 2}(\log d+\log (d+\nu +2\beta+z))-{zd\over 2}\log(d+\nu+2\beta+z+2)\\
&\qquad+{z(d-1)\over d+2\beta+1}(\log(d-1)+\log(\nu+d+1+2\beta+z)-\log(d+2\beta+1))+O\Big({1+|z|^3\over d}\Big)\\
&={z\over 2}\Big((d-1)\log 2+2\log d - {4(\beta+1)\log d\over d+2\beta+1}\Big)+{z^2\over 4}-{z(d-1)\over 2}\Big(2\log d+{\nu+2\beta+z\over d}\Big)\\
&\qquad+{zd\over 2}\Big(2\log d+{\nu +2\beta+z\over d}\Big)-{zd\over 2}\Big(\log d+{\nu+2\beta+z+2\over d}\Big)+O\Big({1+|z|^3\over d}\Big)\\
&={z\over 2}\Big(4\log d - d\log d+(d-1)\log 2-\nu-2\beta-2 -{4(\beta+1)\log d\over d+2\beta+1}\Big)-{z^2\over 4}+O\Big({1+|z|^3\over d}\Big).
\end{align*}
Next, employing the asymptotic behaviour of the logarithm of the gamma function \eqref{eq:AsympGamma} and the Taylor approximation of the logarithm we get
$$
\log\Big({\Gamma({d+1\over 2}+\beta+1)\over\Gamma({d\over 2}+\beta+1)}\Big)=\log\Big({d\over 2}\Big)+O(1/d).
$$
Finally, combining this with \eqref{eq:20.01.21_1} we obtain
\begin{equation}\label{eq:22.01.21}
\begin{aligned}
\log\varphi_d(z)&= z\Big(\Big({9\over 4}+{\nu\over 2}\Big)\log\Big({d\over 2}\Big)-{d+\nu+2\beta\over 2}-1+\log\Big({\pi\over 4\gamma}\Big)-{3(\beta+1)\log d\over d+2\beta+1}-\log(d-1)!\Big)\\
&\qquad+\log\psi(z)+{z^2\over 4}\Big(\log\Big({d\over 2}\Big)-2\Big)+O\Big({1+|z|^3\over d}\Big),
\end{aligned}
\end{equation}
which finishes the proof of the theorem.
\end{proof}

\subsection{Large deviation principle}

In this subsection we complete the investigation of the high-dimensional probabilistic limit theorems for the random variables $Y_{\beta,\nu,d}$ by establishing the large deviation principle for fixed $\beta$ and $\nu$, as $d\to\infty$.

\begin{theorem}[Large deviations for the log-volume]\label{thm:LDP}
For fixed $\nu\ge -1$ and $\beta>-1$ 
the sequence of random variables 
$$
\left({2\over \log\big({d\over 2}\big)}(Y_{\beta,\nu,d}-{\EE Y_{\beta,\nu,d}})\right)_{d\in\NN}
$$ 
satisfies a large deviations principle on $\RR$ with speed ${1\over 2}\log\big({d\over 2}\big)$ and good rate function $I(x)={x^2\over 2}$.
\end{theorem}

Our proof of this result will rely on the G\"artner-Ellis theorem, see \cite[Section 2.3]{DZ}. Although this is a standard tool in the large deviations theory, we reformulate a version of it in order to keep our presentation self-contained.

\begin{lemma}[G\"artner-Ellis theorem]
Consider a sequence of random variables $(X_d)_{d\in\mathbb{N}}$ in $\RR$ with logarithmic moment generating functions $\Lambda_d(t):=\log\EE e^{tX_d}$, $t\in\RR$. Let $(a_d)_{d\in\mathbb{N}}$ be a positive sequence such that $a_d\to\infty$, as $d\to\infty$. Assume that for each $t\in\RR$ the limit
\[
\Lambda(t):=\lim_{d\rightarrow\infty}{1\over a_d}\Lambda_d(a_dt),
\]
exists as an extended real number. Also assume that $D_\Lambda:=\{t\in\RR:\Lambda(t)<\infty\}=\RR$ and that $\Lambda$ is differentiable on $D_\Lambda$. Then the sequence of random variables $X_d$ satisfies large deviations principle with speed $a_d$ and rate function $I(x)=\sup\limits_{t\in\RR}[xt-\Lambda(t)]$, the Legendre-Fenchel transform of $\Lambda$.
\end{lemma}

\begin{proof}[Proof of Theorem \ref{thm:LDP}]
As in the proof of Theorem \ref{thm:ModPhi} we denote by $\varphi_d(z)$ the moment generating function of the random variable $Y_{\beta,\nu,d}$. Then the moment generating function $\widetilde{\varphi}_d(z)$ of the random variable ${2\over \log({d\over 2})}(Y_{\beta,\nu,d}-\EE Y_{\beta,\nu,d})$ satisfies
$$
\log\widetilde{\varphi}_d\Big({t\over 2} \log\Big({d\over 2}\Big)\Big)=\log \varphi_d(t)-z\EE Y_{\beta,\nu,d}.
$$
Using the asymptotic representation \eqref{eq:22.01.21} for the function $\log\varphi_d(t)$ together with Stirling's formula \eqref{eq:Stirling} and Corollary \ref{cor:CumulantsTypicalBeta} we conclude
$$
\lim\limits_{d\to\infty}{2\over \log({d\over 2})}\log\widetilde{\varphi}_d\Big({t\over 2} \log\Big({d\over 2}\Big)\Big)={t^2\over 2}.
$$
Since the Legendre-Fenchel transform  $I(x)$ of the function $t\mapsto t^2/2$ is $x^2/2$ the claim follows from the G{\"a}rtner-Ellis theorem.
\end{proof}

\begin{remark}
The similar theorem has been obtained in \cite[Theorem 5.3]{GusakovaThaeleDelaunay} for the case when $\beta=-1$ or, in other words, for the $\nu$-weighted typical cell of Poisson-Delaunay tessellation. We would like to mention here that the formulation of Theorem 5.3 in \cite{GusakovaThaeleDelaunay} contains a typo. In fact, the rescaling by the factor ${2/\log({n/2})}$ is missing, although it is present in the proof.
\end{remark}

\subsection*{Acknowledgement}
ZK was supported by the DFG under Germany's Excellence Strategy  EXC 2044 -- 390685587, \textit{Mathematics M\"unster: Dynamics - Geometry - Structure}. CT and ZK were supported by the DFG via the priority program \textit{Random Geometric Systems}.

\bibliographystyle{acm}

\end{document}